\documentclass{amsart}
\usepackage{graphicx}
\usepackage{amssymb,amscd,amsthm,amsxtra}
\usepackage{latexsym}
\usepackage{epsfig}

\newtheorem{thm}{Theorem}[section]
\newtheorem{cor}[thm]{Corollary}
\newtheorem{lem}[thm]{Lemma}
\newtheorem{prop}[thm]{Proposition}
\theoremstyle{definition}
\newtheorem{defn}[thm]{Definition}
\theoremstyle{remark}
\newtheorem{rem}[thm]{Remark}
\numberwithin{equation}{section}

\newcommand{\R}{\mathbb R}

\newcommand{\eps}{\varepsilon}

\newcommand{\p}{\partial}

\newcommand{\comment}[1]{}

\begin{document}

\title[Pointwise $C^{2,\alpha}$ estimates]{Pointwise $C^{2,\alpha}$ estimates at the boundary for the Monge-Ampere equation}
\author{O. Savin}
\address{Department of Mathematics, Columbia University, New York, NY 10027}
\email{\tt  savin@math.columbia.edu}

\thanks{The author was partially supported by NSF grant 0701037.}

\begin{abstract}
We obtain pointwise $C^{2,\alpha}$ estimates at boundary points for solutions to the Monge-Ampere equation under appropriate local conditions on the right hand side and boundary data.
\end{abstract}
\maketitle

\section{Introduction}

Boundary estimates for solutions of the Dirichlet problem for the Monge-Ampere equation
$$\left \{
\begin{array}{l}
\det D^2 u=f \quad \mbox{in $\Omega,$}\\
u=\varphi \quad \quad \quad \mbox {on $\p \Omega$,}
\end{array}
\right.
$$
were obtained in the classical paper by Caffarelli, Nirenberg and Spruck \cite{CNS} in the case when $\p \Omega$, $\varphi$ and $f$ are sufficiently smooth. When $f$ is less regular, i.e $f \in C^\alpha$, the global $C^{2,\alpha}$ estimates were obtained by Trudinger and Wang  \cite{TW} for $\varphi$, $\p \Omega \in C^3$. In this paper we discuss pointwise $C^{2,\alpha}$ estimates at boundary points under appropriate local conditions on the right hand side and boundary data. Our main result can be viewed as a natural extension up to the boundary of the pointwise interior $C^{2,\alpha}$ estimate of Caffarelli in \cite{C2}.

We start with the following definition (see \cite{CC}).

{\it Definition:} Let $0< \alpha \le 1$. We say that a function $u$ is {\it pointwise} $C^{2,\alpha}$ at $x_0$ and write $$u \in C^{2, \alpha}(x_0)$$  if there exists a quadratic polynomial $P_{x_0}$ such that $$u(x)=P_{x_0}(x) + O(|x-x_0|^{2+\alpha}).$$
We say that $u\in C^2(x_0)$ if $$u(x)=P_{x_0}(x)+o(|x-x_0|^2).$$
Similarly one can define the notion for a function to be $C^k$ and $C^{k,\alpha}$ at a point for any integer $k \ge 0$.

\

It is easy to check that if $u$ is {\it pointwise} $C^{2, \alpha}$ at all points of a Lipschitz domain $\bar \Omega$ and the equality in the definition above is uniform in $x_0$ then $u \in C^{2, \alpha}(\bar \Omega)$ in the classical sense. Precisely, if there exists $M$ and $\delta$ such that for all points $x_0 \in \bar \Omega$
$$|u(x)-P_{x_0}(x)|\le M|x-x_0|^{2+\alpha}  \quad \mbox{if $|x-x_0|\le \delta$, $\ x \in \bar \Omega$} $$ then $$[D^2 u]_{C^\alpha(\bar \Omega)} \le  C(\delta, \Omega)M.$$

Caffarelli showed in \cite{C2} that if $u$ is a strictly convex solution of $$\det D^2 u=f$$ and  $f \in C^\alpha(x_0)$, $f(x_0)>0$ at some interior point $x_0 \in \Omega$, then $u \in C^{2,\alpha}(x_0)$. Our main theorem deals with the case when $x_0 \in \p \Omega$.

\begin{thm}\label{main}
Let $\Omega$ be a convex domain and let $u:\bar \Omega \to \R$ convex, continuous, solve the Dirichlet problem for the Monge-Ampere equation
  \begin{equation}\label{DP}
  \left \{
\begin{array}{l}
\det D^2 u=f \quad \mbox{in $\Omega,$}\\
u=\varphi \quad \quad \quad \mbox {on $\p \Omega$,}
\end{array}
\right.
\end{equation}
with positive, bounded right hand side i.e $$0< \lambda \le f \le \Lambda,$$ for some constants $\lambda$, $\Lambda$.

Assume that for some point $x_0 \in \p \Omega$ we have $$f \in C^\alpha(x_0), \quad \varphi, \p \Omega \in C^{2, \alpha}(x_0),$$for some $\alpha \in (0,1)$. If $\varphi$ separates quadratically on $\p \Omega$ from the tangent plane of $u$ at $x_0$, then
  $$u \in C^{2, \alpha}(x_0).$$
\end{thm}

The way $\varphi$ separates locally from the tangent plane at $x_0$ is given by the tangential second derivatives of $u$ at $x_0$. Thus the assumption that this separation is quadratic is in fact necessary for the $C^{2, \alpha}$ estimate to hold. Heuristically, Theorem \ref{main} states that if the tangential pure second derivatives of $u$ are bounded below then the boundary Schauder estimates hold for the Monge-Ampere equation.

A more precise, quantitative version of Theorem \ref{main} is given in section 7 (see Theorem \ref{c2alpha}).

Given the boundary data, it is not always easy to check the quadratic separation since it involves some information about the slope of the tangent plane at $x_0$. However, this can be done in several cases (see Proposition \ref{Prop1}). One example is when $\p \Omega$ is uniformly convex and $\varphi$, $\p \Omega \in C^3(x_0)$. The $C^3$ condition of the data is optimal as it was shown by Wang in \cite{W}. Other examples are when $\p \Omega$ is uniformly convex and $\varphi$ is linear, or when $\p \Omega$ is tangent of second order to a plane at $x_0$ and $\varphi$ has quadratic growth near $x_0$.

As a consequence of Theorem \ref{main} we obtain a pointwise $C^{2, \alpha}$ estimate in the case when the boundary data and the domain are pointwise $C^3$. As mentioned above, the global version was obtained by Trudinger and Wang in \cite{TW}.
\begin{thm}\label{m1}
Let $\Omega$ be uniformly convex and let $u$ solve \eqref{DP}. Assume that
$$ f\in C^\alpha(x_0), \quad \varphi, \p \Omega \in C^3(x_0),$$
for some point $x_0 \in \p \Omega$, and some $\alpha \in (0,1)$.
Then $u \in C^{2, \alpha}(x_0).$
\end{thm}

We also obtain the $C^{2,\alpha}$ estimate in the simple situation when $\p \Omega \in C^{2, \alpha}$ and $\varphi$ is constant.

\begin{thm}\label{m3}
Let $\Omega$ be a uniformly convex domain and assume $u$ solves \eqref{DP} with $\varphi \equiv 0$. If $f \in C^\alpha(\bar \Omega)$,  $\p \Omega \in C^{2, \alpha},$ for some $\alpha\in (0,1)$ then $u \in C^{2, \alpha}(\bar \Omega).$
\end{thm}

The key step in the proof of Theorem \ref{main} is a localization theorem for boundary points which was recently proved in \cite{S}. It states that under natural local assumptions on the domain and boundary data, the sections
  $$S_h(x_0)=\{x \in \overline \Omega \ | \quad u(x)<u(x_0)+\nabla u(x_0) \cdot (x-x_0)+h\},$$
with $x_0 \in \p \Omega$ are ``equivalent" to ellipsoids centered at $x_0$.

\begin{thm}\label{loc_thm}
 Let $\Omega$ be convex and $u$ satisfy \eqref{DP}, and assume $$\p \Omega, \varphi \in C^{1,1}(x_0).$$ If $\varphi$ separates quadratically from the tangent plane of $u$ at $x_0$, then for each small $h>0$ there exists an ellipsoid $E_h$ of volume $h^{n/2}$ such that $$cE_h \cap \overline \Omega \subset \, S_h(x_0) -x_0 \subset \, CE_h \cap \overline \Omega,$$
with $c$, $C$ constants independent of $h$.
\end{thm}

Theorem \ref{loc_thm} is an extension up to the boundary of the localization theorem at interior points due to Caffarelli in \cite{C1}. For completeness we provide also its proof in the current paper.

The paper is organized as follows. In section 2 we discuss briefly the compactness of solutions to the Monge-Ampere equation which we use later in the paper (see Theorem \ref{comp}). For this we need to consider also solutions with possible discontinuities at the boundary. In section 3 we give a quantitative version of the Localization Theorem (see Theorem \ref{main_loc}). In sections 4 and 5 we provide the proof of Theorem \ref{main_loc}. In section 6 we obtain a version of the classical Pogorelov estimate in half-domain (Theorem \ref{pe}). Finally, in section 7 we use the previous results together with a standard approximation method and prove our main theorem.

\section{Solutions with discontinuities on the boundary}

Let $u:\Omega \to \R$ be a convex function with $\Omega \subset \R^n$ bounded and convex. Denote by
$$U:=\{(x,x_{n+1}) \in \Omega \times \R | \quad x_{n+1} \ge u(x)\}$$ the upper graph of $u$.

\begin{defn}\label{bv1}
We define the values of $u$ on $\p \Omega$ to be equal to $\varphi$ i.e
$$u|_{\p \Omega}=\varphi,$$ if the upper graph of $\varphi:\p \Omega \to \R \cup\{\infty\}$
$$\Phi:=\{(x,x_{n+1}) \in \p \Omega \times \R | \quad x_{n+1} \ge \varphi(x)\}$$
is given by the closure of $U$ restricted to $\p \Omega \times \R$,
$$\Phi:=\bar U \cap (\p \Omega \times \R ).$$
\end{defn}

From the definition we see that $\varphi$ is {\it lower semicontinuous}.

\

If $u:\Omega \to \R$ is a viscosity solution to $$\det D^2u=f(x),$$ with $f\ge 0$ continuous and bounded on $\Omega$, then there exists an increasing sequence of subsolutions, continuous up to the boundary, $$u_n:\bar \Omega \to \R, \quad \quad \det D^2 u_n \ge f(x)$$ with $$\lim u_n=u \quad \mbox{in $\bar \Omega$},$$ where the values of $u$ on $\p \Omega$ are defined as above.

Indeed, let us assume for simplicity that $0 \in \Omega$, $u(0)=0$, $u \ge 0$. Then, on each ray from the origin $u$ is increasing, hence $v_\eps :\bar \Omega \to \R$, $$v_\eps(x)=u((1-\eps)x)$$ is an increasing family of continuous functions as $\eps \to 0$, with $$\lim v_\eps=u \quad \mbox{in $\bar \Omega$}.$$ In order to obtain a sequence of subsolutions we modify $v_\eps$ as $$u_\eps(x):=v_\eps(x)+w_\eps(x), $$ with $w_\eps \le -\eps$, convex, so that $$\det D^2w_\eps \ge |f(x)-(1-\eps)^{2n}f((1-\eps)x)|,$$ 
thus $$\det D^2u_\eps(x)=\det(D^2v_\eps +D^2w_\eps) \ge \det D^2 v_\eps +\det D^2w_\eps \ge f(x).$$ The claim is proved since as $\eps \to 0$  we can choose $w_\eps$ to converge uniformly to 0.

\begin{prop}[Comparison principle]\label{comp_prin}
Let $u$, $v$ be defined on $\Omega$ with $$\det D^2u \ge f(x) \ge \det D^2v$$ in the viscosity sense and $$u|_{\p \Omega} \le v|_{\p \Omega}.$$ Then $$u \le v \quad \mbox{in $\Omega$}.$$
\end{prop}

\begin{proof}
Since $u$ can be approximated by a sequence of continuous functions on $\bar \Omega$ it suffices to prove the result in the case when $u$ is continuous on $\bar \Omega$ and $u<v$ on $\p \Omega$. Then, $u<v$ in a small neighborhood of $\p \Omega$ and the inequality follows from the standard comparison principle.

\end{proof}

A consequence of the comparison principle is that a solution $\det D^2 u =f$ is determined uniquely by its boundary values $u|_{\p \Omega}$.

Next we define the notion of convergence for functions which are defined on different domains.

\begin{defn}\label{conv_def}
a) Let $u_k:\Omega_k \to \R$ be a sequence of convex functions with $\Omega_k$ convex. We say that $u_k$ converges to $u:\Omega \to \R$ i.e $$u_k \to u  $$ if the upper graphs converge $$\bar U_k \to \bar U \quad \mbox{in the Haudorff distance.}$$ In particular it follows that $\bar \Omega_k \to \bar \Omega$ in the Hausdorff distance.

b) Let $\varphi_k:\p \Omega_k \to \R \cup \{ \infty\}$ be a sequence of lower semicontinuous functions. We say that $\varphi_k$ converges to $\varphi:\p \Omega \to \R \cup \{\infty\}$ i.e $$\varphi_k \to \varphi  $$ if the upper graphs converge $$ \Phi_k \to \Phi \quad \mbox{in the Haudorff distance.}$$

c) We say that $f_k: \Omega_k \to \R$ converge to $f: \Omega \to \R$ if $f_k$ are uniformly bounded and $$f_k \to f$$ uniformly on compact sets of $\Omega$.

\end{defn}

{\it Remark:} When we restrict the Hausdorff distance to the nonempty closed sets of a compact set we obtain a compact metric space. Thus, if $\Omega_k$, $u_k$ are uniformly bounded then we can always extract a convergent subsequence $u_{k_m} \to u$. Similarly, if $\Omega_k$, $\varphi_k$ are uniformly bounded we can extract a convergent subsequence $\varphi_{k_m} \to \varphi.$

\

\begin{prop}\label{con_env}
Let $u_k:\bar \Omega_k \to \R$ be continuous and  $$\det D^2u_k=f_k, \quad u_k=\varphi_k \quad \mbox{on $\p \Omega_k$}.$$ If $$u_k \to u, \quad \varphi_k \to \varphi, \quad f_k \to f,$$ then
\begin{equation}\label{DP1}
\det D^2u=f, \quad u=\varphi^* \quad \mbox{on $\p \Omega$},
 \end{equation}
 where $\varphi^*$ is the convex envelope of $\varphi$ on $\p \Omega$ i.e $\Phi^*$ is the restriction to $\p \Omega \times \R$ of the convex hull generated by $\Phi$.
\end{prop}

{\it Remark:} If $\Omega$ is strictly convex then $\varphi^*=\varphi$.

\begin{proof}
Since $$\bar U_k \to \bar U, \quad \Phi_k \to \Phi, \quad \Phi_k \subset \bar U_k,$$ we see that $\Phi \subset \bar U$. Thus, if $K$ denotes the convex hull generated by $\Phi$, then $\Phi^* \subset K \subset \bar U$. It remains to show that $\bar U \cap (\p \Omega \times \R) \subset K.$

Indeed consider a hyperplane $$x_{n+1} = l(x)$$
which lies below $K.$ Then $$u_k - l \geq 0 \quad \text{on $\p \Omega_k$}$$ and by Alexandrov estimate we have that
$$u_k - l \geq -Cd_k^{1/n}$$ where $d_k$ represents the distance to $\p \Omega_k.$ By taking $k \rightarrow \infty$ we see that
$$u - l \geq -C d^{1/n}$$ thus no point on $\p \Omega \times \R$ below the hyperplane belongs to $\bar U.$

\end{proof}

Proposition \ref{con_env} says that given any $\varphi$ bounded and lower semicontinuous, and $f\ge0$ bounded and continuous we can always solve uniquely the Dirichlet problem

$$\left \{
\begin{array}{l}
\det D^2 u=f \quad \mbox{in $\Omega,$}\\
u=\varphi \quad \quad \quad \mbox {on $\p \Omega$}
\end{array}
\right.
$$
by approximation. Indeed, we can find sequences $\varphi_k$, $f_k$ of continuous, uniformly bounded functions defined on strictly convex domains $\Omega_k$ such that $\varphi_k \to \varphi$ and $f_k \to f$. Then the corresponding solutions $u_k$ are uniformly bounded and continuous up to the boundary. Using compactness and the proposition above we see that $u_k$ must converge to the unique solution $u$ in \eqref{DP1}.

We extend the Definition \ref{bv1} in order to allow a boundary data that is not necessarily convex.

\begin{defn}\label{bv2}
Let $\varphi:\p \Omega \to \R$ be a lower semicontinuous function. When we write that a convex function $u$ satisfies
$$u=\varphi \quad \mbox{on $\p \Omega$}$$ we understand $$u|_{\p \Omega}=\varphi^*$$ where $\varphi^*$ is the convex envelope of $\varphi$ on $\p \Omega$.
\end{defn}

Whenever $\varphi^*$ and $\varphi$ do not coincide we can think of the graph of $u$ as having a vertical part on $\p \Omega$ between $\varphi^*$ and $\varphi$.

It follows easily from the definition above that the boundary values of $u$ when we restrict to the domain $$\Omega_h:=\{u<h\}$$ are given by
$$\varphi_h=\varphi \quad \mbox{on}\quad \p \Omega \cap \{\varphi \le h\} \subset \p \Omega_h$$ and $\varphi_h=h$ on the remaining part of $\p \Omega_h$.

\

By Proposition \ref{comp_prin}, the comparison principle still holds. Precisely, if $$u=\varphi, \quad v=\psi, \quad \varphi \le \psi \quad \quad \mbox{on $\p \Omega$,}$$
$$ \det D^2 u \ge f \ge \det D^2 v \quad \mbox{in $\Omega$},$$
then $$u \le v \quad \mbox{in $\Omega$.}$$

The advantage of introducing the notation of Definition \ref{bv2} is that the boundary data is preserved under limits.

\begin{prop}\label{comp0}
Assume $$\det D^2 u_k=f_k, \quad u_k=\varphi_k \quad \mbox{on $\p \Omega_k$},$$ with $\Omega_k$, $\varphi_k$ uniformly bounded and $$\varphi_k \to \varphi, \quad f_k \to f.$$ Then $$u_k \to u$$ and $u$ satisfies $$\det D^2u=f, \quad u=\varphi \quad \mbox{on $\p \Omega$.}$$
\end{prop}

\begin{proof} Using the compactness of solutions we may assume that $u_k$ converges to a limit $u$ and it remains to prove that $u=\varphi$ on $\p \Omega$.

Denote by $\tilde u_k$ the restriction of $u_k$ to the set $$\tilde \Omega_k=\{x \in \Omega_k| \quad dist(x, \p \Omega_k) \ge \eps_k\}.$$ Notice that for fixed $k$, as $\eps_k \to 0$ then $\tilde u_k \to u_k$ and $\tilde \varphi_k \to \varphi_k^*$. On the other hand, from the hypotheses we obtain that $\varphi_k^* \to \varphi^*$. Thus, we can choose a sequence of $\eps_k \to 0$ such that $$\tilde u_k \to u, \quad \tilde \varphi_k \to \varphi^*, \quad \tilde f_k \to f.$$ Now, since $\tilde u_k$ are continuous up to the boundary, the conclusion follows from Proposition \ref{con_env}.
\end{proof}

Finally, we state a version of the last proposition for solutions with bounded right-hand side i.e
$$ \lambda \le \det D^2u \le \Lambda,$$
where the two inequalities are understood in the viscosity sense.

\begin{thm}\label{comp}
Assume $$\lambda \le \det D^2u_k \le \Lambda, \quad u_k=\varphi_k \quad \mbox{on $\p \Omega_k$},$$ and $\Omega_k$, $\varphi_k$ uniformly bounded.

Then there exists a subsequence $k_m$ such that $$u_{k_m} \to u, \quad \varphi_{k_m} \to \varphi$$ with
$$\lambda \le \det D^2u \le \Lambda, \quad u=\varphi \quad \mbox{on $\p \Omega$}.$$

\end{thm}

\section{The Localization Theorem}

In this section we state the quantitative version of the localization theorem at boundary points (Theorem \ref{main_loc}).

Let $\Omega$ be a bounded convex set in $\R^n$. We assume that
\begin{equation}\label{om_ass}
B_\rho(\rho e_n) \subset \, \Omega \, \subset \{x_n \geq 0\} \cap B_{\frac 1\rho},
\end{equation}
for some small $\rho>0$, that is $\Omega \subset (\R^n)^+$ and $\Omega$ contains an interior ball tangent to $\p \Omega$ at $0.$

Let $u : \overline \Omega \rightarrow \R$ be continuous, convex, satisfying
\begin{equation}\label{eq_u}
\det D^2u =f, \quad \quad 0<\lambda \leq f \leq \Lambda \quad \text{in $\Omega$}.
\end{equation}
We extend $u$ to be $\infty$ outside $\overline \Omega.$

After subtracting a linear function we assume that
\begin{equation}\label{eq_u1}
\mbox{$x_{n+1}=0$ is the tangent plane to $u$ at $0$,}
\end{equation}
in the sense that $$u \geq 0, \quad u(0)=0,$$
and any hyperplane $x_{n+1}= \eps x_n$, $\eps>0$, is not a supporting plane for $u$.

We investigate the geometry of the sections of $u$ at $0$ that we denote for simplicity of notation
$$S_h := \{x \in \overline \Omega : \quad u(x) < h \}.$$

We show that if the boundary data has quadratic growth near $\{x_n=0\}$ then, as $h \rightarrow 0$, $S_h$ is equivalent to a half-ellipsoid centered at 0.

Precisely, our theorem reads as follows.

\begin{thm}[Localization Theorem]\label{main_loc} Assume that $\Omega$, $u$ satisfy \eqref{om_ass}-\eqref{eq_u1} above and for some $\mu>0$,
\begin{equation}\label{commentstar}\mu |x|^2 \leq u(x) \leq \mu^{-1} |x|^2 \quad \text{on $\p \Omega \cap \{x_n \leq \rho\}.$}\end{equation}
Then, for each $h<c(\rho)$ there exists an ellipsoid $E_h$ of volume $h^{n/2}$ such that
$$kE_h \cap \overline \Omega \, \subset \, S_h \, \subset \, k^{-1}E_h \cap \overline \Omega.$$

Moreover, the ellipsoid $E_h$ is obtained from the ball of radius $h^{1/2}$ by a
linear transformation $A_h^{-1}$ (sliding along the $x_n=0$ plane)
$$A_hE_h= h^{1/2}B_1$$
$$A_h(x) = x - \nu x_n, \quad \nu = (\nu_1, \nu_2, \ldots, \nu_{n-1}, 0), $$
with
$$ |\nu| \leq k^{-1} |\log h|.$$
The constant $k$ above depends on $\mu, \lambda, \Lambda, n$ and $c(\rho)$ depends also on $\rho$.
\end{thm}

 The ellipsoid $E_h$, or equivalently the linear map $A_h$, provides information about the behavior of the second derivatives near the origin. Heuristically, the theorem states that in $S_h$ the tangential second derivatives are bounded from above and below and the mixed second derivatives are bounded by $|\log h|$.

The hypothesis that $u$ is continuous up to the boundary is not necessary, we just need to require that \eqref{commentstar} holds in the sense of Definition \ref{bv2}.

Given only the boundary data $\varphi$ of $u$ on $\p \Omega$, it is not always easy to check the main assumption \eqref{commentstar} i.e that $\varphi$ separates quadratically on $\p \Omega$ (in a neighborhood of $\{x_n=0\}$) from the tangent plane at $0$. Proposition \ref{Prop1} provides some examples when this is satisfied depending on the local behavior of $\p \Omega$ and $\varphi$ (see also the remarks below).

\begin{prop}\label{Prop1}
Assume \eqref{om_ass},\eqref{eq_u} hold. Then \eqref{commentstar} is satisfied if any of the following holds:

\

\noindent
1) $\varphi$ is linear in a neighborhood of $0$ and $\Omega$ is uniformly convex at the origin.

\

\noindent
2) $\p \Omega$ is tangent of order 2 to $\{x_n=0\}$ and $\varphi$ has quadratic growth in a neighborhood of $\{x_n=0\}$.

\

\noindent
3) $\varphi$, $\p \Omega \in C^3(0)$, and $\Omega$ is uniformly convex at the origin.

\end{prop}

Proposition \ref{Prop1} is standard (see \cite{CNS}, \cite{W}). We sketch its proof below.

\begin{proof}
1) Assume $\varphi=0$ in a neighborhood of $0$. By the use of standard barriers, the assumptions on $\Omega$ imply that the tangent plane at the origin is given by
$$x_{n+1}=-\mu x_n$$
for some bounded $\mu>0$. Then \eqref{commentstar} clearly holds.

2) After subtracting a linear function we may assume that $$ \mu|x'|^2 \le \varphi \le \mu^{-1}|x'|^2$$ on $\p \Omega$ in a neighborhood of $\{x_n=0\}$. Using a barrier we obtain that $l_0$, the tangent plane at the origin, has bounded slope. But $\p \Omega$ is tangent of order 2 to $\{x_n=0\}$, thus $l_0$ grows less than quadratic on $\p \Omega$ in a neighborhood of $\{x_n=0\}$ and \eqref{commentstar} is again satisfied.

3) Since $\Omega$ is uniformly convex at the origin, we can use barriers and obtain that $l_0$ has bounded slope. After subtracting this linear function we may assume $l_0=0$. Since $\varphi$, $\p \Omega \in C^3(0)$ we find that
$$\varphi=Q_0(x') + o(|x'|^3)$$ with $Q_0$ a cubic polynomial. Now $\varphi \ge 0$, hence $Q_0$ has no linear part and its quadratic part is given by, say
$$\sum_{i<n} \frac {\mu_i}{2} x_i^2  , \quad \mbox{with} \quad \mu_i \ge 0.$$
We need to show that $\mu_i>0$.

If $\mu_1=0$, then the coefficient of $x_1^3$ is $0$ in $Q_0$. Thus, if we restrict to $\p \Omega$ in a small neighborhood near the origin, then for all small $h$ the set $\{\varphi <h \}$ contains $$ \{|x_1| \le r(h)h^{1/3}\} \cap \{ |x'| \le c h^{1/2} \} $$
for some $c>0$ and with $$r(h) \to \infty \quad \mbox{as $h \to 0$}.$$ Now $S_h$ contains the convex set generated by $\{\varphi <h\}$ thus, since $\Omega$ is uniformly convex,
$$|S_h| \ge c'(r(h)h^{1/3})^3 h^{(n-2)/2} \ge c' r(h)^3 h^{n/2}.$$
On the other hand, since $u$ satisfies \eqref{eq_u} and $$0 \le u \le h \quad \mbox{in $S_h$}$$ we obtain (see \eqref{s_h_upbd}) $$|S_h| \le C h^{n/2},$$ for some $C$ depending on $\lambda$ and $n$, and we contradict the inequality above as $ h \to 0$.

\end{proof}

\begin{rem}\label{rem0}The proof easily implies that if $\p \Omega$, $\varphi \in C^3(\Omega)$ and $\Omega$ is uniformly convex,  then we can find a constant $\mu$ which satisfies \eqref{commentstar} for all $x \in \p \Omega$.
\end{rem}

\begin{rem}\label{rem1} From above we see that we can often verify \eqref{commentstar} in the case when $\varphi$, $\p \Omega \in C^{1,1}(0)$ and $\Omega$ is uniformly convex at $0$. Indeed, if $l_\varphi$ represents the tangent plane at $0$ to $\varphi:\p \Omega \to \R$ (in the sense of \eqref{eq_u1}), then \eqref{commentstar} holds if either $\varphi$ separates from $l_\varphi$ quadratically near $0$, or if $\varphi$ is tangent to $l_\varphi$ of order 3 in some tangential direction.
\end{rem}

\begin{rem}\label{rem2}
Given $\varphi$, $\p \Omega \in C^{1,1}(0)$ and $\Omega$ uniformly convex at $0$, then \eqref{commentstar} holds if $\lambda$ is sufficiently large.
\end{rem}

\section{Proof of Theorem \ref{main_loc} (I)}

We prove Theorem \ref{main_loc} in the next two sections. In this section we obtain some preliminary estimates and reduce the theorem to a statement about the rescalings of $u$. This statement is proved in section 5 using compactness.

Next proposition was proved by Trudinger and Wang in \cite{TW}. It states that the volume of $S_h$ is proportional to $h^{n/2}$ and after an affine transformation (of controlled norm) we may assume that the center of mass of $S_h$ lies on the $x_n$ axis. Since our setting is slightly different we provide its proof.

\begin{prop}\label{TW} Under the assumptions of Theorem \ref{main_loc}, for all $h \leq c(\rho),$ there exists a linear transformation (sliding along $x_n=0$) $$A_h(x) = x - \nu x_n,$$ with$$ \nu_n=0, \quad |\nu|\leq C(\rho) h^{-\frac{n}{2(n+1)}}$$ such that the rescaled function
$$\tilde u(A_h x) = u(x),$$ satisfies in $$\tilde S_h := A_h S_h = \{\tilde u<h\}$$ the following:
\begin{enumerate}
\item the center of mass of $\tilde S_h$ lies on the $x_n$-axis;
\item $$k_0 h^{n/2} \leq |\tilde S_h| = |S_h| \leq k_0^{-1} h^{n/2};$$
\item the part of $\p \tilde S_h$ where $\{\tilde u <h\}$ is a graph, denoted by $$\tilde G_h = \p \tilde S_h \cap \{\tilde u <h\} = \{(x', g_h(x'))\}$$ that satisfies $$g_h \leq C(\rho)|x'|^2$$ and $$\frac \mu 2 |x'|^2 \leq \tilde u \leq 2\mu^{-1} |x'^2| \quad \text{on $\tilde G_h$}.$$
\end{enumerate}

The constant $k_0$ above depends on $\mu, \lambda, \Lambda, n$ and the constants $C(\rho), c(\rho)$ depend also on $\rho$.
\end{prop}

\

 In this section we denote by $c$, $C$ positive constants that depend on $n$, $\mu$, $\lambda$, $\Lambda$. For simplicity of notation, their values may change from line to line whenever there is no possibility of confusion. Constants that depend also on $\rho$ are denote by $c(\rho)$, $C(\rho)$.

\begin{proof}
The function
$$v: = \mu |x'|^2 + \frac{\Lambda}{\mu^{n-1}} x_n^2 -C(\rho) x_n$$ is a lower barrier for $u$ in $\Omega \cap \{x_n \leq \rho\}$ if $C(\rho)$ is chosen large.

Indeed, then $$v \leq u \quad \text{on $\p \Omega \cap \{x_n \leq \rho\}$},$$

$$v \leq 0 \leq u \quad \text {on $\Omega \cap \{x_n=\rho\}$},$$
and
$$\det D^2 v > \Lambda.$$
In conclusion, $$v \leq u \quad \text{in $\Omega \cap \{x_n \leq \rho\}$},$$ hence
\begin{equation}\label{star} S_h \cap \{x_n \leq \rho\} \subset \{v <h\} \subset \{x_n > c(\rho)(\mu |x'|^2- h)\}.
\end{equation}

Let $x^*_h$ be the center of mass of $S_h.$ We claim that
\begin{equation}\label{2star}x^*_h \cdot e_n \geq c_0(\rho) h^{\alpha}, \quad \alpha= \frac{n}{n+1},\end{equation}
 for some small $c_0(\rho)>0$.

Otherwise, from \eqref{star} and John's lemma we obtain
$$S_h \subset \{x_n \leq C(n) c_0 h^{\alpha} \leq h^\alpha\} \cap \{|x'| \leq C_1 h^{\alpha/2}\},$$ for some large $C_1=C_1(\rho)$. Then the function
$$w= \eps x_n + \frac{h}{2} \left(\frac{|x'|}{C_1h^{\alpha/2}}\right)^2 + \Lambda C_1 ^{2(n-1)} h \left(\frac{x_n}{h^\alpha}\right)^2$$
is a lower barrier for $u$ in $S_h$ if $c_0$ is sufficiently small.

Indeed,
$$w \leq \frac h 4 + \frac h 2 + \Lambda C_1 ^{2(n-1)}(C(n)c_0 )^2 h < h \quad \text{in $S_h,$}$$
and for all small $h$,
$$w \leq \eps x_n + \frac{h^{1-\alpha}}{C_1 ^2} |x'|^2 + C(\rho)hc_0\frac{x_n}{h^\alpha} \leq \mu |x'|^2 \leq u \quad \text{on $\p \Omega$,}$$
and
$$\det D^2 w = 2\Lambda.$$

Hence $$w \leq u \quad \text{in $S_h$,}$$ and we contradict that 0 is the tangent plane at 0. Thus claim \eqref{2star} is proved.

Now, define $$A_h x = x - \nu x_n, \quad \nu = \frac{x^{*'}_h}{x_h^* \cdot e_n},$$
and
$$\tilde u(A_h x) = u(x).$$

The center of mass of $\tilde S_h=A_hS_h$ is $$\tilde x ^*_h=A_hx^*_h$$ and lies on the $x_n$-axis from the definition of $A_h$.
Moreover, since $x^*_h \in S_h$, we see from \eqref{star}-\eqref{2star} that
$$|\nu| \leq C(\rho) \frac{(x_h^*\cdot e_n)^{1/2}}{(x_h^*\cdot e_n)} \leq C(\rho) h^{-\alpha/2},$$
and this proves (i).

If we restrict the map $A_h$ on the set on $\p \Omega$ where $\{u < h\}$, i.e. on
$$\p S_h \cap \p \Omega \subset \{x_n \leq \frac{|x'|^2}{\rho}\} \cap \{|x'| < Ch^{1/2}\}$$
we have
$$|A_h x - x| = |\nu| x_n \leq C(\rho) h^{-\alpha/2} |x'|^2  \leq C(\rho) h^{\frac{1-\alpha}{2}} |x'|,$$
and part (iii) easily follows.

Next we prove (ii). From John's lemma, we know that after relabeling the $x'$ coordinates if necessary,
\begin{equation}\label{3star} D_h B_1 \subset \tilde S_h - \tilde x^*_h \subset C(n) D_h B_1\end{equation}where
\[
 D_h =
 \begin{pmatrix}
  d_1 & 0 & \cdots & 0 \\
  0 & d_{2} & \cdots & 0 \\
  \vdots  & \vdots  & \ddots & \vdots  \\
  0 & 0 & \cdots & d_{n}
 \end{pmatrix}.
\]

Since $$\tilde u \leq 2 \mu^{-1}|x'|^2 \quad \text{on $\tilde G_h = \{(x', g_h(x'))\}$},$$ we see that the domain of definition of $g_h$ contains a ball of radius $(\mu h/2)^{1/2}$. This implies that
$$d_i \geq c_1 h^{1/2}, \quad \quad i=1,\cdots, n-1,$$ for some $c_1$ depending only on $n$ and $\mu.$ Also from \eqref{2star} we see that $$\tilde x^*_h \cdot e_n =x^*_h \cdot e_n \ge c_0(\rho) h^\alpha$$ which gives $$d_n \ge c(n) \tilde x^*_h \cdot e_n \ge c(\rho) h^\alpha.$$

We claim that for all small $h$,
$$\prod_{i=1}^n d_i \geq k_0 h^{n/2},$$
with $k_0$ small depending only on $\mu, n, \Lambda,$ which gives the left inequality in (ii).

To this aim we consider the barrier,
$$w= \eps x_n + \sum_{i=1}^n ch\left(\frac{x_i}{d_i}\right)^2.$$ We choose $c$ sufficiently small depending on $\mu, n, \Lambda$ so that for all $h<c(\rho)$,
$$w \leq h \quad \text{on $\p \tilde S_h$,}$$
and on the part of the boundary $\tilde G_h$, we have $w \le \tilde u$ since
\begin{align*}w & \leq \eps x_n+\frac{c}{c_1^2}|x'|^2 + c h \left(\frac {x_n}{d_n}\right)^2 \\
 & \leq \frac \mu 4 |x'|^2 + c h C(n) \frac{x_n}{d_n} \\
 & \leq \frac \mu 4 |x'|^2 + c h^{1-\alpha}C(\rho)|x'|^2\\
 &\le \frac \mu 2 |x'|^2.
 \end{align*}
Moreover, if our claim does not hold, then
$$\det D^2 w = (2c h )^n (\prod d_i)^{-2n} > \Lambda,$$ thus $w \le \tilde u$ in $\tilde S_h$. By definition, $\tilde u$ is obtained from $u$ by a sliding along $x_n=0$, hence $0$ is still the tangent plane of $\tilde u$ at $0$. We reach again a contradiction since $\tilde u \ge w\ge \eps x_n$ and the claim is proved.

Finally we show that
\begin{equation}\label{s_h_upbd}
|\tilde S_h| \leq Ch^{n/2}
 \end{equation}
 for some $C$ depending only on $\lambda, n.$ Indeed, if $$v=h \quad \text{on $\p \tilde S_h$},$$ and $$\det D^2v= \lambda$$ then
$$v \geq u \geq 0 \quad \text{in $\tilde S_h$.}$$ Since
$$h \geq h-\min_{\tilde S_h} v \geq c(n,\lambda) |\tilde S_h|^{2/n}$$ we obtain the desired conclusion.

\end{proof}

\

In the proof above we showed that for all $h \leq c(\rho),$ the entries of the diagonal matrix $D_h$ from \eqref{3star} satisfy
$$d_i \geq c h^{1/2}, \quad i=1,\ldots n-1$$

$$d_n \geq c(\rho)h^{\alpha}, \quad \alpha= \frac{n}{n+1}$$

$$c h^{n/2} \leq \prod d_i \leq Ch^{n/2}.$$

The main step in the proof of Theorem \ref{main} is the following lemma that will be completed in Section 5.

\begin{lem}\label{l1}
There exist constants $c$, $c(\rho)$ such that
\begin{equation}\label{dn}d_n \geq ch^{1/2},\end{equation}
for all $h \le c(\rho)$.
\end{lem}

Using Lemma \ref{l1} we can easily finish the proof of our theorem.

\

{\it Proof of Theorem \ref{main}.} Since all $d_i$ are bounded below by $c h^{1/2}$ and their product is bounded above by $Ch^{n/2}$ we see that
$$C h^{1/2} \geq d_i \geq ch^{1/2} \quad \quad i=1,\cdots,n$$ for all $h\le c(\rho)$. Using \eqref{3star} we obtain $$\tilde S_h \subset Ch^{1/2}B_1.$$ Moreover, since $$\tilde x^*_h \cdot e_n \ge d_n \ge c h^{1/2}, \quad \quad  (\tilde x^*_h)'=0,$$ and the part $\tilde G_h$ of the boundary $\p \tilde S_h$ contains the graph of $\tilde g_h$ above $|x'| \le c h^{1/2}$, we find that $$ch^{1/2}B_1 \cap \tilde \Omega \subset \tilde S_h,$$ with $\tilde \Omega=A_h \Omega$, $\tilde S_h=A_h S_h$. In conclusion
$$ch^{1/2}B_1 \cap \tilde \Omega \subset A_h S_h \subset Ch^{1/2}B_1.$$
We define the ellipsoid $E_h$ as
$$E_h:=A_h^{-1}(h^{1/2}B_1),$$
hence $$c E_h \cap \overline \Omega \subset S_h \subset C E_h.$$
Comparing the sections at levels $h$ and $h/2$ we find
$$cE_{h/2} \cap \overline \Omega \subset C E_h$$
and we easily obtain the inclusion
$$ A_hA_{h/2}^{-1} B_1 \subset C B_1.$$
If we denote $$A_hx=x-\nu_h x_n$$ then the inclusion above implies $$|\nu_h-\nu_{h/2}| \le C,$$ which gives the desired bound $$|\nu_h| \le C|\log h|$$ for all small $h$.

\qed

\

In order to prove Lemma \ref{l1} we introduce a new quantity $b(h)$ which is proportional to $d_n h^{-1/2}$ and is appropriate when dealing with affine transformations.

\

\textbf{Notation.} Given a convex function $u$ we define $$b_u(h) = h^{-1/2} \sup_{S_h} x_n.$$ Whenever there is no possibility of confusion we drop the subindex $u$ and use the notation $b(h)$.

\

Below we list some basic properties of $b(h)$.

\

1) If $h_1 \le h_2$ then $$\left(\frac{h_1}{h_2}\right)^\frac 12 \le \frac{b(h_1)}{b(h_2)} \le \left(\frac{h_2}{h_1}\right)^\frac 12.$$

2) A rescaling $$ \tilde u (Ax) =u(x)$$ given by a linear transformation $A$ which leaves the $x_n$ coordinate invariant does not change the value of $b,$ i.e $$b_{\tilde u}(h)=b_u(h).$$

3) If $A$ is a linear transformation which leaves the plane $\{x_n=0\}$ invariant the values of $b$ get multiplied by a constant. However the quotients $b(h_1)/b(h_2)$ do not change values i.e $$\frac{b_{\tilde u}(h_1)}{b_{\tilde u}(h_2)}=\frac{b_u(h_1)}{b_u(h_2)}.$$

4) If we multiply $u$ by a constant, i.e.
$$\tilde u(x) = \beta u(x)$$
then $$ b_{\tilde u}(\beta h)= \beta^{-1/2}b_u(h),$$ and $$ \frac{b_{\tilde u}(\beta h_1)}{b_{\tilde u}(\beta h_2)}=\frac{b_u(h_1)}{b_u(h_2)}.$$

\

From \eqref{3star} and property 2 above,
$$c(n)d_n \leq b(h)h^{1/2} \leq C(n)d_n,$$
hence Lemma \ref{l1} will follow if we show that $b(h)$ is bounded below. We achieve this by proving the following lemma.

 \begin{lem}\label{l2} There exist $c_0$, $c(\rho)$ such that if $h \le c(\rho)$ and $b(h) \le c_0$ then \begin{equation}\label{quo}\frac{b(t h)}{b(h)} >2,
 \end{equation} for some $t \in [c_0, 1]$.
 \end{lem}
  This lemma states that if the value of $b(h)$ on a certain section is less than a critical value $c_0$, then we can find a lower section at height still comparable to $h$ where the value of $b$ doubled. Clearly Lemma \ref{l2} and property 1 above imply that $b(h)$ remains bounded for all $h$ small enough.

The quotient in \eqref{quo} is the same for $\tilde u$ which is defined in Proposition \ref{TW}. We normalize the domain $\tilde S_h$ and $\tilde u$ by considering the rescaling
$$v(x)= \frac{1}{h} \tilde u(h^{1/2}Ax)$$
where $A$ is a multiple of $D_h$ (see \eqref{3star}), $A=\gamma D_h$ such that $$\det A=1.$$ Then $$ch^{-1/2} \le \gamma \le C h^{-1/2},$$ and the diagonal entries of $A$ satisfy $$a_i \ge c , \quad \quad i=1,2,\cdots, n-1,$$ $$ c b_u(h) \le a_n \le Cb_u(h).$$

The function $v$ satisfies
$$\lambda \leq \det D^2v \leq \Lambda,$$

$$v \geq 0, \quad  v(0)=0,$$
is continuous and it is defined in $\bar \Omega_v$ with

$$\Omega_v:= \{v<1\} = h^{-1/2}A^{-1}\tilde S_h.$$ Then

$$x^* + cB_1 \subset \Omega_v \subset C B_1^+,$$ for some $x^*$, and

$$ct^{n/2} \leq |S_t(v)| \leq Ct^{n/2}, \quad \forall t\leq 1,$$
where $S_t(v)$ denotes the section of $v$. Since $$\tilde u=h \quad \mbox{in} \quad \p \tilde S_h \cap \{x_n \ge C(\rho) h\},$$ then
$$v=1 \quad \text{on $\p \Omega_v \cap \{x_n \geq \sigma\}, \quad \sigma:=C(\rho)h^{1-\alpha}$}.$$
Also, from Proposition \ref{TW} on the part $G$ of the boundary of $\p \Omega_v$ where $\{v<1\}$ we have
 \begin{equation}\label{ai}
 \frac 1 2 \mu \sum_{i=1}^{n-1} a_i^2 x_i^2 \leq v \leq 2 \mu^{-1} \sum_{i=1}^{n-1} a_i^2 x_i^2.
 \end{equation}

In order to prove Lemma \ref{l2} we need to show that if $\sigma$, $a_n$ are sufficiently small depending on $n, \mu, \lambda, \Lambda$ then the function $v$ above satisfies
\begin{equation}\label{v}
b_v (t) \geq 2b_v(1)
\end{equation} for some $1 > t \geq c_0.$

Since $\alpha<1$, the smallness condition on $\sigma$ is satisfied by taking $h<c(\rho)$ sufficiently small. Also $a_n$ being small is equivalent to one of the $a_i$, $1 \le i \le n-1$ being large since their product is 1 and $a_i$ are bounded below.

In the next section we prove property \eqref{v} above by compactness, by letting $\sigma \to 0$, $a_i \to \infty$ for some $i$ (see Proposition \ref{HDprop}).

\section{Proof of Theorem \ref{main_loc} II}

 In this section we consider the class of solutions $v$ that satisfy the properties above. After relabeling the constants $\mu$ and $a_i$, and by abuse of notation writing $u$ instead of $v$, we may assume we are in the following situation.

Fix $\mu$ small and $\lambda, \Lambda.$
For an increasing sequence $$a_1\leq a_2\leq \ldots \leq a_{n-1}$$
with $$a_1 \ge \mu,$$
we consider the family of solutions
$$u \in \mathcal{D}_\sigma^\mu(a_1, a_2, \ldots, a_{n-1})$$
of convex functions $u : \Omega \rightarrow \R$ that satisfy
\begin{equation}\label{HD1} \lambda \leq \det D^2u \leq \Lambda \quad \text{in $\Omega$,} \quad \text{$0 \le u \le 1$ in $\Omega$};
\end{equation}
\begin{equation}\label{HD2}0 \in \p\Omega, \quad B_{\mu}(x_0) \subset \Omega \subset B_{1/\mu}^+ \quad \text{for some $x_0$;}
\end{equation}
\begin{equation}\label{HD3} \mu h^{n/2} \leq |S_h| \leq \mu^{-1} h^{n/2}.
\end{equation}
 Moreover we assume that the boundary $\p \Omega$ has a closed subset $G$
 \begin{equation}\label{HD3.1}G \subset \{x_n \le \sigma\} \cap \p \Omega
  \end{equation}
  which is a graph in the $e_n$ direction with projection $\pi_n(G)\subset \R^{n-1}$ along $e_n$
\begin{equation}\label{HD3.2}
\{\, \mu^{-1}   \sum_{1}^{n-1} a_i^2 x_i^2 \le 1  \, \} \subset \pi_n(G) \subset  \{ \, \mu   \sum_{1}^{n-1} a_i^2 x_i^2 \le 1   \, \},
\end{equation}
and (see Definition \ref{bv2}), the boundary values of $u=\varphi$ on $\p \Omega$ satisfy
\begin{equation}\label{HD4}
\varphi =1 \quad \text{on $\p \Omega \setminus G$}; \end{equation}
and
\begin{equation}\label{HD5}  \mu \sum_{1}^{n-1} a_i^2 x_i^2  \leq \varphi \leq  \min \{ \, 1, \quad  \mu^{-1} \sum_{1}^{n-1} a_i^2 x_i^2 \, \} \quad \quad \text{on $G$.}
\end{equation}

In this section we prove

\begin{prop}\label{HDprop} For any $M>0$ there exists $C_*$ depending on $M, \mu, \lambda, \Lambda,n$ such that if $u \in \mathcal{D}_\sigma^\mu(a_1, a_2, \ldots, a_{n-1})$ with $$a_{n-1} \geq C_*, \quad \sigma \leq C_*^{-1}$$ then $$b(h)= (\sup_{S_h} x_n) h^{-1/2} \geq M$$ for some $h$ with $ C_*^{-1} \leq h \leq 1 .$\end{prop}

Property \eqref{v} (hence Theorem \ref{main_loc}), easily follows from this proposition. Indeed, by choosing $$M= 2 \mu^{-1} \ge 2b(1)$$ in Proposition \ref{HDprop} we prove the existence of a section $S_h$ with $h \geq c_0$ such that $$b(h) \geq 2 b(1).$$ Clearly the function $v$ of the previous section satisfies the hypotheses above (after renaming the constant $\mu$) provided that $\sigma$, $a_n$ are sufficiently small.

We prove Proposition \ref{HDprop} by compactness. We introduce the limiting solutions from the class $\mathcal{D}_\mu^\sigma(a_1,\ldots,a_{n-1})$ when $a_{k+1} \rightarrow \infty$ and $\sigma \rightarrow 0.$

If $\mu \le a_1 \le \ldots \le a_k$, we denote by $$\mathcal{D}_0^\mu(a_1,\ldots, a_k, \infty, \infty, \ldots, \infty) , \quad 0 \le k \le n-2$$ the class of functions $u$ that satisfy properties \eqref{HD1}-\eqref{HD2}-\eqref{HD3} with,
\begin{equation}
G \subset \{x_i=0, \quad i > k\} \cap \p \Omega
\end{equation}
and if we restrict to the space generated by the first $k$ coordinates then
\begin{equation}
\{\, \mu^{-1}   \sum_{1}^k a_i^2 x_i^2 \le 1  \, \} \subset G \subset  \{ \, \mu   \sum_{1}^k a_i^2 x_i^2 \le 1   \, \}.
\end{equation}
Also, $u=\varphi$ on $\p \Omega$ with
\begin{equation}\label{4'} \varphi=1 \quad \text{on $\p \Omega \setminus G$};\end{equation}
\begin{equation}  \mu \sum_{1}^k a_i^2 x_i^2  \leq \varphi \leq  \min \{ \, 1, \quad  \mu^{-1} \sum_{1}^k a_i^2 x_i^2  \, \} \quad \quad \text{on $G$.}
\end{equation}
The compactness theorem (Theorem \ref{comp}) implies that if $$u_m \in D_{\sigma_m}^\mu(a^m_1,\ldots, a^m_{n-1})$$ is a sequence with $$\sigma_m \to 0 \quad \mbox{ and } \quad a^m_{k+1} \to \infty$$ for some fixed $0 \le k \le n-2$, then we can extract a convergent subsequence to a function $u$ (see Definition \ref{conv_def}) with
$$u \in D_0^\mu(a_1,..,a_l,\infty,..,\infty),$$
for some $l \le k$ and $a_1 \le \ldots \le a_l.$

Proposition \ref{HDprop} follows easily from the next proposition.

\begin{prop}\label{HDprop1} For any $M>0$ and $0 \leq k \leq n-2$ there exists $c_k$ depending on $M, \mu, \lambda, \Lambda,n, k$ such that if
\begin{equation}\label{d_0}u \in \mathcal{D}_0^\mu(a_1, \ldots,a_k, \infty, \ldots, \infty)
 \end{equation}
 then $$b(h)= (\sup_{S_h} x_n) h^{-1/2} \geq M$$ for some $h$ with $ c_k \leq h \leq 1 .$
\end{prop}

Indeed, if Proposition \ref{HDprop} fails for a sequence of constants $C_* \to \infty$ then we obtain a limiting solution $u$ as in \eqref{d_0} for which $b(h) \le M$ for all $h>0$. This contradicts Proposition \ref{HDprop1} (with $M$ replaced by $2M$).

\

We prove Proposition \ref{HDprop1} by induction on $k$. We start by introducing some notation.

Denote $$x= (y,z,x_n), \quad y=(x_1, \ldots, x_k) \in \R^k, \quad z=(x_{k+1}, \ldots, x_{n-1}) \in \R^{n-1-k}.$$

\begin{defn}\label{sliding}
We say that a linear transformation $T:\R^n \to \R^n$ is a {\it sliding along the $y$ direction} if
$$Tx := x+ \nu_1 z_1+ \nu_2 z_2+ \ldots + \nu_{n-k-1} z_{n-k-1} + \nu_{n-k}x_n$$ with $$\nu_1, \nu_2, \ldots, \nu_{n-k} \in span\{e_1, \ldots, e_k\}$$
\end{defn}

\

 We see that $T$ leaves the $(z,x_n)$ components invariant together with the subspace $(y,0,0)$. Clearly, if $T$ is a sliding along the $y$ direction then so is $T^{-1}$ and $$\det T=1.  $$

The key step in the proof of Proposition \ref{HDprop1} is the following lemma.

\begin{lem}\label{lastlem}
Assume that $$u \geq p(|z| - q x_n),$$ for some $p,q>0$ and assume that for each section $S_h$ of $u$, $h\in (0,1)$, there exists $T_h$ a sliding along the $y$ direction such that $$T_hS_h \subset C_0 h^{1/2} B_1^+,$$ for some constant $C_0$. Then
$$u \, \not \in \, D_0^\mu(1,\ldots,1, \infty, \ldots, \infty).$$

\end{lem}

\begin{proof}
Assume by contradiction that $u\in D_0^\mu$ and it satisfies the hypotheses with $q \le q_0$ for some $q_0$. We show that
\begin{equation}\label{5.1}
u \geq p'(|z| - q' x_n), \quad \quad q'=q-\eta,
\end{equation}
 for some $0<p' \ll p,$ where the constant $\eta>0$ depends only on $q_0$ and $\mu, C_0, \Lambda, n$.

Then, since $q'\le q_0$, we can apply this result a finite number of times and obtain
$$u \ge \eps(|z|+x_n),$$ for some small $\eps>0$. This gives $S_h \subset \{x_n \le \eps^{-1} h\}$ hence $$T_h S_h \subset \{x_n \le \eps^{-1} h\}  $$ and by the hypothesis above $$|S_h|=|T_hS_h|=O(h^{(n+1)/2}) \quad \mbox{as $h \to 0$},$$ and we contradict \eqref{HD3}.

Now we prove \eqref{5.1}. Since $u\in D_0^\mu$ as above, there exists a closed set $$G_h \subset \p S_h \cap \{z=0,x_n=0\}$$ such that on the subspace $(y,0,0)$
$$\{\mu^{-1}|y|^2 \le h\} \subset G_h \subset \{\mu |y|^2 \le h\},$$
and the boundary values $\varphi_h$ of $u$ on $\p S_h$ satisfy (see Section 2)
$$ \varphi_h=h \quad \text{on $\p S_h \setminus G_h$};$$
$$ \mu |y|^2  \leq \varphi_h \leq  \min \left\{h, \mu^{-1}|y|^2  \right \} \quad \quad \text{on $G_h$.}$$

Let $w$ be a rescaling of $u$, $$w(x):= \frac{1}{h} u(h^{1/2} T_h^{-1} x)$$ for some small $h\ll a.$  Then $$S_1(w) := \Omega_w=h^{-1/2}T_h S_h \subset B^+_{C_0}$$ and our hypothesis becomes
\begin{equation}\label{w} w \geq \frac{p}{h^{1/2}} (|z| -q x_n).
 \end{equation} Moreover the boundary values $\varphi_w$ of $w$ on $\p \Omega_w$ satisfy
 $$\varphi_w=1 \quad \text{on $\p \Omega_w \setminus G_w$}$$
 $$\mu |y|^2 \le \varphi_w \le \min\{1, \mu^{-1}|y|^2\} \quad \mbox{on} \quad G_w:=h^{-1/2}G_h.$$

Next we show that $\varphi_w \ge v$ on $\p \Omega_w$ where $v$ is defined as
$$v := \delta |x|^2 + \frac{\Lambda}{\delta^{n-1}}(z_1-qx_n)^2 + N(z_1-qx_n) +\delta x_n, $$ and $\delta$ is small depending on $\mu$ and $C_0$, and $N$ is chosen large such that $$\frac{\Lambda}{\delta^{n-1}} t^2 + Nt$$ is increasing in the interval $|t|\le (1+q_0)C_0.$

From the definition of $v$ we see that $$\det D^2v > \Lambda.$$

On the part of the boundary $\p \Omega_w $ where $z_1 \le qx_n$ we use that $\Omega_w \subset B_{C_0}$ and obtain
$$v \le \delta (|x|^2+x_n) \le \varphi_w.$$

On the part of the boundary $\p \Omega_w $ where $z_1 > qx_n$ we use \eqref{w} and obtain
$$1=\varphi_w \geq C(|z| -q x_n) \ge C(z_1-qx_n)$$ with $C$ arbitrarily large provided that $h$ is small enough. We choose $C$ such that
the inequality above implies $$\frac{\Lambda}{\mu^{n-1}}(z_1-qx_n)^2 + N(z_1 -qx_n) <\frac 1 2.$$ Then
$$\varphi_w=1 > \frac 12 + \delta (|x|^2+x_n) \ge v.$$

In conclusion $\varphi_w \ge v$ on $\p \Omega_w$ hence the function $v$ is a lower barrier for $w$ in $\Omega_w$. Then
$$w \geq N(z_1 -qx_n)+\delta x_n$$ and, since this inequality holds for all directions in the $z$-plane, we obtain  $$w \geq N(|z| -(q-\eta)x_n), \quad \quad \eta := \frac \delta N.$$ Scaling back we get $$u \ge p'(|z| -(q-\eta)x_n) \quad \quad \mbox{in $S_h$}.$$ Since $u$ is convex and $u(0)=0$, this inequality holds globally, and \eqref{5.1} is proved.

\end{proof}

\begin{lem}\label{base} Proposition \ref{HDprop1} holds for $k=0$.\end{lem}
\begin{proof} By compactness we need to show that there does not exist $u \in \mathcal{D}_0^\mu(\infty, \ldots, \infty)$ with $b(h) \leq M$ for all $h$. If such $u$ exists then $G=\{0\}$. Let
$$v:= \delta (|x'| + \frac 1 2 |x'|^2) + \frac{\Lambda}{\delta^{n-1}} x_n^2 - N x_n$$ with $\delta$ small depending on $\mu$, and $N$ large so that
$$\frac {\Lambda} {\delta^{n-1}} x_n^2 - Nx_n \le 0$$ in $B_{1/\mu}^+.$
Then
$$v \leq \varphi \quad \text{on $\p \Omega$}, \quad \det D^2 v > \Lambda,$$
hence $$v \leq u \quad \text{in $\Omega$}.$$ This gives
$$u \geq \delta |x'| - Nx_n,$$
and we obtain $$S_h \subset \{|x'| \le C(x_n+h)\}.$$ Since $b(h) \le M$ we conclude
$$S_h \subset Ch^{1/2}B_1^+,$$
and we contradict Lemma \ref{lastlem} for $k=0$.

\end{proof}

Now we prove Proposition \ref{HDprop1} by induction on $k$.

\

\textit{Proof of Proposition \ref{HDprop1}.} In this proof we denote by $c$, $C$ positive constants that depend on $M, \mu, \lambda, \Lambda, n$ and $k$.

We assume that the proposition holds for all nonnegative integers up to $k-1$, $1 \le k<n-2$, and we prove it for $k.$ Let $$u \in D_0^\mu(a_1,\ldots,a_k, \infty,\ldots, \infty).$$
By the induction hypotheses and compactness we see that there exists a constant $$C_k(\mu, M, \lambda, \Lambda,n)$$ such that if $a_k \ge C_k$ then $b(h) \ge M$ for some $h \ge C_k^{-1}$. Thus, it suffices to consider only the case when $a_k < C_k$.

If no $c_{k+1}$ exists then we can find a limiting solution that, by abuse of notation, we still denote by $u$ such that
\begin{equation}\label{mubar}
u \in \mathcal{D}_0^{\tilde \mu}(1,1,\ldots, 1, \infty, \ldots, \infty)
\end{equation}
with \begin{equation}\label{HDstar}b(h) \le M h^{1/2}, \quad \forall h>0\end{equation} where $\tilde \mu $ depends on $\mu$ and $C_k.$

We show that such a function $u$ does not exist.

Denote as before $$x= (y,z,x_n), \quad y=(x_1, \ldots, x_k) \in \R^k, \quad z=(x_{k+1}, \ldots, x_{n-1}) \in \R^{n-1-k}.$$
On $\p \Omega$ we have
$$\varphi(x) \geq \delta |x'|^2 + \delta |z| + \frac{\Lambda}{\delta^{n-1}} x_n^2 - N x_n$$ where $\delta$ is small depending on $\tilde \mu$, and $N$ is large so that $$ \frac{\Lambda}{\delta^{n-1}} x_n^2 - N x_n \le 0 $$ in $B_{1/\tilde \mu}^+.$ As before we obtain that the inequality above holds in $\Omega$, hence
\begin{equation}\label{HD2star}u(x) \geq \delta |z| - N x_n.\end{equation}
From \eqref{HDstar}-\eqref{HD2star} we see that the section $S_h$ of $u$ satisfies
\begin{equation}\label{HD3star}S_h \subset \{|z| < \delta^{-1}(Nx_n + h)\} \cap \{x_n \leq Mh^{1/2}\}.\end{equation}

From John's lemma we know that $S_h$ is equivalent to an ellipsoid $E_h$ of the same volume i.e
\begin{equation}\label{JL}
c(n)E_h \subset S_h-x_h^* \subset C(n) E_h, \quad |E_h|=|S_h|,
\end{equation}
with $x_h^*$ the center of mass of $S_h$.

For any ellipsoid $E_h$ in $\R^n$ of positive volume we can find $T_h$, a sliding along the $y$ direction (see Definition \ref{sliding}), such that
\begin{equation}\label{T_h}
T_h E_h = |E_h|^{1/n}A B_1,
\end{equation} with a matrix $A$ that leaves the $(y,0,0)$ and $(0,z,x_n)$ subspaces invariant, and $\det A=1$.
By choosing an appropriate system of coordinates in the $y$ and $z$ variables we may assume in fact that
 $$A(y,z,x_n) = (A_1 y, A_2(z,x_n))$$ with
\[
 A_{1} =
 \begin{pmatrix}
  \beta_{1} & 0 & \cdots & 0 \\
  0 & \beta_{2} & \cdots & 0 \\
  \vdots  & \vdots  & \ddots & \vdots  \\
  0 & 0 & \cdots & \beta_{k}
 \end{pmatrix}
\]
with $0<\beta_1 \le \cdots \le \beta_k$, and

\[
 A_{2} =
 \begin{pmatrix}
  \gamma_{k+1} & 0 & \cdots & 0 &\theta_{k+1} \\
  0 & \gamma_{k+2} & \cdots & 0 & \theta_{k+2} \\
  \vdots  & \vdots  & \ddots & \vdots & \vdots  \\
  0 & 0 & \cdots & \gamma_{n-1} & \theta_{n-1}\\
  0 & 0 & \cdots & 0& \theta_{n}
 \end{pmatrix}
\]
with $\gamma_j$, $\theta_n >0$.

The $h$ section $\tilde S_h=T_hS_h$ of the rescaling
$$\tilde u(x) = u(T_h^{-1}x)$$ satisfies \eqref{HD3star} and since $u \in \mathcal{D}_0^\mu$, there exists $\tilde G_h=G_h$,
$$\tilde G_h\subset \{z=0,x_n=0\} \cap \p \tilde S_h$$ such that on the subspace $(y,0,0)$
$$\{\mu^{-1}|y|^2 \le h\} \subset \tilde G_h \subset \{\mu |y|^2 \le h\},$$
and the boundary values $\tilde \varphi_h$ of $\tilde u$ on $\p \tilde S_h$ satisfy
$$ \tilde \varphi_h=h \quad \text{on $\p \tilde S_h \setminus \tilde G_h$};$$
$$ \mu |y|^2  \leq \tilde \varphi_h \leq  \min \left\{h, \mu^{-1}|y|^2  \right \} \quad \quad \text{on $\tilde G_h$.}$$
Moreover, using that $$|S_h| \sim h^{n/2}$$ in \eqref{JL}, \eqref{T_h} and that $0 \in \p S_h$, we obtain
 \begin{equation}\label{S_h}
 \tilde x_h^* + c h^{1/2} AB_1 \subset \tilde S_h \subset C h^{1/2} AB_1, \quad \det A=1,
 \end{equation}
for the matrix $A$ as above and with $\tilde x_h^*$ the center of mass of $\tilde S_h$.

 Next we use the induction hypothesis and show that $\tilde S_h$ is equivalent to a ball.

\begin{lem} \label{ball}
There exists $C_0$ such that $$ T_h S_h=\tilde S_h \subset C_0 h^{n/2} B_1^+.$$
 \end{lem}

\begin{proof}
We need to show that $$|A| \le C.$$

Since $\tilde S_h$ satisfies \eqref{HD3star} we see that
$$\tilde S_h \subset \{ |(z,x_n)| \le C h^{1/2}\},$$
which together with the inclusion \eqref{S_h} gives $|A_2| \le C$ hence
$$\gamma_j, \theta_n \leq C, \quad |\theta_j| \leq C.$$
Also, since $$\tilde G_h \subset \tilde S_h,$$ we find from \eqref{S_h} $$\beta_i \geq c >0, \quad  i=1,\cdots,k.$$

We define the rescaling
$$w(x) = \frac 1 h \tilde u (h^{1/2} Ax)$$ defined in a domain $\Omega_w=S_1(w)$. Then \eqref{S_h} gives
$$B_c(x_0) \subset \Omega_w \subset B^+_C, $$
and $w=\varphi_w$ on $\p \Omega_w$ with
$$\varphi_w=1 \quad \text{on $\p \Omega_w \setminus G_w$},$$
$$\tilde \mu \sum_1^k \beta_i^2 x_i^2 \leq \varphi_w \leq \min \{1, \tilde \mu^{-1} \sum_1^k \beta_i^2 x_i^2\} \ \quad \text{on $G_w:=h^{-1/2}A^{-1}\tilde G_h$}.$$
This implies that $$w \in \mathcal{D}^{\bar \mu}_0(\beta_1, \beta_2, \ldots, \beta_k, \infty, \ldots, \infty)$$ for some small $\bar {\mu}$ depending on $\mu, M, \lambda, \Lambda, n,k$.

We claim that $$b_u(h) \ge c_\star.$$ First we notice that $$b_u(h)=b_{\tilde u}(h) \sim \theta_n.$$

Since
$$\theta_n \prod \beta_i \prod \gamma_j  =\det A=1$$ and $$\gamma_j \leq C,$$ we see that if $b_u(h)$ (and therefore $\theta_n$) becomes smaller than a critical value $c_*$ then
$$\beta_k \geq C_k(\bar \mu, \bar M, \lambda, \Lambda,n),$$
with $\bar M:=2 \bar \mu^{-1}$, and by the induction hypothesis
$$b_w(\tilde h) \geq \bar M \ge 2 b_w(1)$$ for some $\tilde h > C_k^{-1}$. This gives
$$\frac{b_u(h \tilde h)}{b_u(h)}=\frac{b_w(\tilde h)}{b_w(1)} \ge 2,$$
which implies $b_u(h \tilde h) \ge 2 b_u(h)$ and our claim follows.

Next we claim that $\gamma_j$ are bounded below by the same argument. Indeed, from the claim above $\theta_n$ is bounded below and if some $\gamma_j$ is smaller than a small value $\tilde c_*$ then
$$\beta_k \geq C_k(\bar \mu, \bar M_1,\lambda, \Lambda,n)$$ with $$\bar M_1:=\frac{2M}{\bar \mu c_\star}.$$ By the induction hypothesis $$b_w(\tilde h) \geq \bar M_1 \geq \frac{2M}{c_\star} b_w(1),$$ hence
$$\frac{b_u(h \tilde h)}{b_u(h)} \geq \frac{2M}{c_\star}$$ which gives $b_u(h \tilde h) \ge 2M$, contradiction.
In conclusion $\theta_n$, $\gamma_j$ are bounded below which implies that $\beta_i$ are bounded above. This shows that $|A|$ is bounded and the lemma is proved.

\end{proof}

{\it End of the proof of Proposition \ref{HDprop1}.}

The proof is finished since Lemma \ref{ball}, \eqref{mubar}, \eqref{HD2star} contradict Lemma \ref{lastlem}.

\qed

\section{Pogorelov estimate in half-domain}

In this section we obtain a version of Pogorelov estimate at the boundary (Theorem \ref{pe} below). A similar estimate was proved also in \cite{TW}. We start with the following a priori estimate.

\begin{prop}\label{pe0}
Let $u:\bar \Omega \to \R$, $u \in C^4(\bar \Omega)$ satisfy the Monge-Ampere equation $$\det D^2 u=1 \quad \mbox{in $\Omega$}.$$ Assume that for some constant $k>0$,
$$B_k^+ \subset \Omega \subset B^+_{k^{-1}},$$
and
$$\left \{
\begin{array}{l}
u=\frac 12 |x'|^2 \quad \mbox{on} \quad \p \Omega \cap \{x_n=0\}\\
u=1 \quad \quad \mbox{on} \quad \p \Omega \cap \{x_n>0\}.
\end{array}
\right.
$$
Then
$$\|u\|_{C^{3,1}(\{u<\frac{1}{16}k^2\})} \le C(k,n).$$
\end{prop}

\begin{proof} We divide the proof into four steps.

{\it Step 1:} We show that $$|\nabla u|\le C(k,n) \quad \mbox{in the set} \quad D:=\{u< k^2 /2 \}.$$

For each $$x_0 \in \{|x'|\le k, \quad x_n=0\},$$ we consider the barrier
$$w_{x_0}(x):=\frac 1 2 |x_0|^2 + x_0 \cdot (x-x_0) + \delta |x'-x_0|^2 + \delta^{1-n}(x_n^2-k^{-1}x_n),$$
where $\delta$ is small so that
$$w_{x_0} \le 1 \quad \mbox{in} \quad B^+_{k^{-1}}.$$ Then
$$w_{x_0} (x_0) = u (x_0), \quad w_{x_0} \le u \quad \mbox{on} \quad \p \Omega \cap \{x_n=0\},$$
$$w_{x_0} \le 1 =u  \quad \mbox{on} \quad \p \Omega \cap \{x_n>0\}, $$
and
$$\det D^2 w_{x_0} >1,$$
thus in $\Omega$
$$u \ge w_{x_0} \ge u(x_0) + x_0 \cdot (x-x_0) -\delta^{1-n}k^{-1}x_n.$$

This gives a lower bound for $u_n(x_0)$. Moreover, 
writing the inequality for all $x_0$ with $|x_0|=k$ we obtain $$D \subset \{x_n \ge c(|x'|-k) \}.$$ 
From the values of $u$ on $\{x_n=0\}$ and the inclusion above 
we obtain a lower bound on $u_n$ on $\p D $ in a neighborhood of $\{x_n=0\}$. Since $\Omega$ contains the cone generated by $k e_n$ and $\{|x'|\le 1, x_n=0\}$ and $u \le 1$ in $\Omega$, we can use the convexity of $u$ and obtain also an upper bound for $u_n$ and all $|u_i|$, $1\le i \le n-1$,  on $\p D $ in a neighborhood of $\{x_n=0\}$. We find $$|\nabla u| \le C \quad \mbox{on} \quad \p D \cap \{x_n \le c_0\},$$ where $c_0>0$ is a small constant depending on $k$ and $n$.
We obtain a similar bound on $\p D \cap \{x_n \ge c_0\}$ by bounding below $$dist(\p D \cap \{x_n \ge c_0\}, \p \Omega)$$ by a small positive constant.
Indeed, if $$y \in \p \Omega \cap \{x_n \ge c_0/2\},$$ then there exists a linear function $l_y$ with bounded gradient so that
$$u(y)=l_y(y), \quad u \ge l_y \quad \mbox{on} \quad \p \Omega.$$ Then, using Alexandrov estimate for $(u-l_y)^-$ we obtain $$u(x) \ge l_y(x)-Cd(x)^{1/n}, \quad \quad d(x):=dist(x,\p \Omega)$$ hence $D$ stays outside a fixed neighborhood of $y$.

\

{\it Step 2:} We show that $$\|D^2 u\| \le C(k,n) \quad \mbox{on} \quad E:=\{x_n=0\} \cap \{|x'|\le k/2\}.$$

It suffices to prove that $|u_{in}|$ are bounded in $E$ with $i=1,..,n-1$. Let $$L\, \varphi:=u^{ij} \varphi_{ij}$$ denote the linearized Monge-Ampere operator for $u$. Then
\begin{align*}
L\,u_i&=0, \quad \quad u_i=x_i \quad \mbox{on} \quad \{x_n=0\},\\
L\,u&=n,
\end{align*}
and if we define $P(x)=\delta |x'|^2+\delta^{1-n}x_n^2$ then
\begin{align*}
L\, P&=Tr\left ((D^2u)^{-1}D^2P \right) \\
& \ge n \left( \det (D^2u)^{-1} \det D^2P\right)^\frac 1n \\
& \ge n.
\end{align*}

Fix $x_0 \in E$. We compare $u_i$ and
$$v_{x_0}(x):=x_i + \gamma_1 \left [ \delta|x'-x_0|^2+\delta^{1-n}(x_n^2-\gamma_2 x_n) - (u-l_{x_0})\right ],$$
where $l_{x_0}$ denotes the supporting linear function for $u$ at $x_0$, $\delta=1/4$, and $\gamma_1$, $\gamma_2 \ge 0$. Clearly, $$L \, v_{x_0} \ge 0,$$ and, since $u$ is Lipschitz in $D$ we can choose $\gamma_1$, $\gamma_2$ large, depending only on $k$ and $n$ such that
$$v_{x_0} \le u_i \quad \mbox{on}\quad \p D.$$
This shows that the inequality above holds also in $D$ and we obtain a lower bound on $u_{in}(x_0)$. Similarly we obtain an upper bound.

\

{\it Step 3:} We show that $$\|D^2 u\| \le C \quad \mbox{on} \quad \{u<  k^2/8\}.$$ We apply the classical Pogorelov estimate in the set $$F:=\{u<k^2/4 \}.$$ Precisely if the maximal value of $$\log \left ( \frac 1 4 k^2 -u \right ) + \log u_{ii} + \frac 1 2 u_i^2$$ occurs in the interior of $F$ then this value is bounded by a constant depending only on $n$ and $\max_F|\nabla u|$ (see \cite{C2}). From step 2, the expression is bounded above on $\p F$ and the estimate follows.

\

{\it Step 4:} The Monge-Ampere equation is uniformly elliptic in $\{u < k^2 / 8 \}$ and by Evans-Krylov theorem and Schauder estimates we obtain the desired $C^{3,1}$ bound.

\end{proof}

\begin{rem}\label{rem61} Assume the boundary values of $u$ are given by $$\left \{
\begin{array}{l}
u=p(x') \quad \mbox{on} \quad \p \Omega \cap \{x_n=0\}\\
u=1 \quad \quad \mbox{on} \p \Omega \cap \{x_n>0\},
\end{array}
\right.
$$
with $p(x')$ a quadratic polynomial that satisfies $$\rho |x'|^2 \le p(x') \le \rho^{-1} |x'|^2,$$
for some $\rho>0$. Then
$$\|u\|_{C^{3,1}(\{u<\frac{1}{16}k^2\})} \le C(\rho,k,n).$$

Indeed, after an affine transformation we can reduce the problem to the case $p(x')=|x'|^2/2$.
\end{rem}

\begin{rem}\label{rem62} Proposition \ref{pe0} holds as well if we replace the half-space $\{x_n \ge 0\}$ with a large ball of radius $\eps^{-1}$
$$ \mathcal{B}_\eps:=\{ \, |x-\eps^{-1}e_n| \le \eps^{-1}\}.$$

Precisely, if
$$ B_k \cap \mathcal{B}_\eps \subset \Omega \subset B_{k^{-1}} \cap \mathcal{B}_\eps,$$
and the boundary values of $u$ satisfy
$$\left \{
\begin{array}{l}
u=\frac 12 |x'|^2 \quad \mbox{on} \quad B_1 \cap \p \mathcal{B}_\eps \subset \p \Omega \\
u \in [1,2] \quad \quad \mbox{on} \quad \p \Omega \setminus (B_1 \cap \p \mathcal{B}_\eps),
\end{array}
\right.
$$
then for all small $\eps$,
$$\|u\|_{C^{3,1}(\{u<k^2/16 \})} \le C,$$
with $C$ depending only on $k$ and $n$.

The proof is essentially the same except that in the barrier functions $w_{x_0}$, $v_{x_0}$ we need to replace $x_n$ by $(x-x_0)\cdot \nu_{x_0}$ where $\nu_{x_0}$ denotes the inner normal to $\p \Omega$ at $x_0$, and in step 2 we work (as in \cite{CNS}) with the tangential derivative
$$T_i:=(1-\eps x_n)\p _{x_i}+\eps x_i \p_{x_n},$$
instead of $\p_{x_i}$.
\end{rem}

As a consequence of the Proposition \ref{pe0} and the remarks above we obtain

\begin{thm}\label{pe}
Let $u: \Omega \to \R$ satisfy the Monge-Ampere equation $$\det D^2 u=1 \quad \mbox{in $\Omega$}.$$ Assume that for some constants $\rho,k>0$,
$$B_k^+ \subset \Omega \subset B^+_{k^{-1}},$$
and (see Definition \ref{bv2}) the boundary values of $u$ are given by
$$\left \{
\begin{array}{l}
u=p(x') \quad \mbox{on} \quad \{p(x') \le 1\} \cap \{x_n=0\} \subset \p \Omega\\
u=1 \quad \quad \mbox{on the rest of $\p \Omega$,}
\end{array}
\right.
$$
where $p$ is a quadratic polynomial that satisfies $$\rho|x'|^2 \le p(x') \le \rho^{-1}|x'|^2.$$

Then
\begin{equation}\label{c4b}
\|u\|_{C^{3,1}(B^+_{c_0})} \le c_0^{-1},
\end{equation}
with $c_0>0$ small, depending only on $k$, $\rho$ and $n$.
\end{thm}

\begin{proof}
We approximate $u$ on $\p \Omega$ by a sequence of smooth functions $u_m$ on $\p \Omega_m$, with $\Omega_m$ smooth, uniformly convex, so that $u_m$, $\Omega_m$ satisfy the conditions of Remark \ref{rem62} above. Notice that $u_m$ is smooth up to the boundary by the results in \cite{CNS}, thus we can use Proposition \ref{pe0} for $u_m$. We let $m \to \infty$ and obtain \eqref{c4b} since $$B_{c_0}^+ \subset \{u<k^2/16\},$$ by convexity.

\end{proof}

\section{Pointwise $C^{2, \alpha}$ estimates at the boundary}

Let $\Omega$ be a bounded convex set with
\begin{equation}\label{om_ass6}
B_\rho(\rho e_n) \subset \, \Omega \, \subset \{x_n \geq 0\} \cap B_{\frac 1\rho},
\end{equation}
for some small $\rho>0$, that is $\Omega \subset (\R^n)^+$ and $\Omega$ contains an interior ball tangent to $\p \Omega$ at $0.$

Let $u : \overline \Omega \rightarrow \R$ be convex, continuous, satisfying
\begin{equation}\label{eq_u6}
\det D^2u =f, \quad \quad 0 <\lambda \leq f \leq \Lambda \quad \text{in $\Omega$},
\end{equation}
and
\begin{equation}\label{eq_u16}
\mbox{$x_{n+1}=0$ is a tangent plane to $u$ at $0$,}
\end{equation}
in the following sense:
$$u \geq 0, \quad u(0)=0,$$
and any hyperplane $x_{n+1}= \eps x_n$, $\eps>0$ is not a supporting plane for $u$.

We also assume that on $\p \Omega$, in a neighborhood of $\{x_n=0\}$, $u$ separates quadratically from the tangent plane $\{x_{n+1}=0\}$,
\begin{equation}\label{commentstar6}\rho |x|^2 \leq u(x) \leq \rho^{-1} |x|^2 \quad \text{on $\p \Omega \cap \{x_n \leq \rho\}.$}\end{equation}

Our main theorem is the following.

\begin{thm}\label{c2alpha}
Let $\Omega$, $u$ satisfy \eqref{om_ass6}-\eqref{commentstar6} above with $f \in C^\alpha$ at the origin, i.e $$|f(x)-f(0)| \le M|x|^\alpha \quad \mbox{in} \quad  \Omega \cap B_\rho,$$
for some $M>0$, and $\alpha \in (0,1)$. Suppose that $\p \Omega$ and $u|_{\p \Omega}$ are $C^{2, \alpha}$ at the origin, i.e we assume that
on $\p \Omega \cap B_\rho$ we satisfy
$$ |x_n-q(x')|\le M|x'|^{2+\alpha},$$
$$ |u-p(x')|\le M|x'|^{2+\alpha},$$
where $p(x')$, $q(x')$ are quadratic polynomials.

Then $u \in C^{2, \alpha}$ at the origin, that is there exists a quadratic polynomial $\mathcal{P}_0$ with
$$\det D^2 \mathcal{P}_0=f(0), \quad \|D^2 \mathcal{P}_0\| \le C(M),$$
such that
$$|u-\mathcal{P}_0|\le C(M)|x|^{2+\alpha} \quad \mbox{in} \quad  \Omega \cap B_\rho,$$ where $C(M)$ depends on $M$, $\rho$, $\lambda$, $\Lambda$, $n$, $\alpha$.
\end{thm}

From \eqref{om_ass6} and \eqref{commentstar6} we see that $p$, $q$ are homogenous of degree 2 and $$\|D^2 p\|, \|D^2 q\| \le \rho^{-1}.$$

A consequence of the proof of Theorem \ref{c2alpha} is that if $f\in C^\alpha$ near the origin, then $u \in C^{2, \alpha}$ in any cone $\mathcal{C}_\theta$ of opening $\theta < \pi /2$ around the $x_n$-axis i.e
$$\mathcal{C}_\theta:=\{x \in (\R^n)^+| \quad |x'| \le x_n \tan \theta \}.$$

\begin{cor}\label{cor1}
Assume $u$ satisfies the hypotheses of Theorem \ref{c2alpha} and $$\|f\|_{C^\alpha(\bar \Omega)} \le M.$$ Given any $\theta < \pi /2$ there exists $\delta(M, \theta)$ small, such that
$$\|u\|_{C^{2, \alpha}(\mathcal{C}_\theta \cap B_\delta)} \le C(M, \theta).$$
\end{cor}

We also mention the global version of Theorem \ref{c2alpha}.

\begin{thm}\label{global}
Let $\Omega$ be a bounded, convex domain and let $u : \overline \Omega \rightarrow \R$ be convex, Lipschitz continuous, satisfying
$$\det D^2u =f, \quad \quad 0<\lambda \leq f \leq \Lambda \quad \text{in $\Omega$}.$$
Assume that
$$\p \Omega, \quad u|_{\p \Omega} \in C^{2,\alpha}, \quad \quad f \in C^\alpha(\bar \Omega), $$ for some $\alpha \in (0,1)$ and
there exists a constant $\rho>0$ such that $$u(y)-u(x)-\nabla u(x) \cdot (y-x) \ge \rho |y-x|^2 \quad \quad \forall x,y \in \p \Omega,$$
where $\nabla u(x)$ is understood in the sense of \eqref{eq_u16}. Then $u \in C^{2,\alpha}(\bar \Omega)$ and
$$\|u\|_{C^{2,\alpha}(\bar \Omega)} \le C,$$
with $C$ depending on $\|\p \Omega\|_{C^{2,\alpha}}$, $\| u|_{\p \Omega}\|_ { C^{2,\alpha}}$, $\|u\|_{C^{0,1}(\bar \Omega)}$, $\|f\|_{C^\alpha(\bar \Omega)}$, $\rho$, $\lambda$, $\Lambda$, $n$, $\alpha$.
\end{thm}

In general, the Lipschitz bound is easily obtained from the boundary data $u|_{\p \Omega}$. We can always do this if for example $\Omega$ is uniformly convex. 

The proof of Theorem \ref{c2alpha} is similar to the proof of the interior $C^{2,\alpha}$ estimate from \cite{C2}, and it  has three steps. First we use the localization theorem to show that after a rescaling it suffices to prove the theorem only in the case when $M$ is arbitrarily small (see Lemma \ref{lem1}). Then we use Pogorelov estimate in half-domain (Theorem \ref{pe}) and reduce further the problem to the case when $u$ is arbitrarily close to a quadratic polynomial (see Lemma \ref{lem2}). In the last step we use a standard iteration argument to show that $u$ is well-approximated by quadratic polynomials at all scales.

We assume for simplicity that $$f(0)=1,$$ otherwise we divide $u$ by $f(0)$.

Constants depending on $\rho$, $\lambda$, $\Lambda$, $n$ and $\alpha$ are called universal. We denote them by $C$, $c$ and they may change from line to line whenever there is no possibility of confusion. Constants depending on universal constants and other parameters i.e M, $\sigma$, $\delta$, etc. are denoted as $C(M,\sigma, \delta)$.

We denote linear functions by $l(x)$ and quadratic polynomials which are homogenous and convex we denote by $p(x')$, $q(x')$, $P(x)$.

The localization theorem says that the section $S_h$ is comparable to an ellipsoid $E_h$ which is obtained from $B_{h^{1/2}}$ by a sliding along $\{x_n=0\}$. Using an affine transformation we can normalize $S_h$ so that it is comparable to $B_1$. In the next lemma we show that, if $h$ is sufficiently small, the corresponding rescaling $u_h$ satisfies the hypotheses of $u$ in which the constant $M$ is replaced by an arbitrary small constant $\sigma$.

\begin{lem}\label{lem1}
 Given any $\sigma>0$, there exist small constants $h=h_0(M,\sigma)$, $k>0$ depending only on $\rho$, $\lambda$, $\Lambda$, $n$, and a rescaling of $u$
 $$u_h(x):=\frac{u(h^{1/2}A^{-1}_hx)}{h}$$ where $A_h$ is a linear transformation with $$\det A_h=1, \quad \|A_h^{-1}\|, \,\|A_h\| \le k^{-1} |\log h|,$$ so that

 a) $$B_k \cap \bar \Omega_h \subset S_1(u_h) \subset B_{k^{-1}}^+, \quad \quad S_1(u_h):=\{u_h < 1\},$$

 b) $$\det D^2 u_h=f_h, \quad \quad |f_h(x)-1| \le \sigma |x|^\alpha \quad \mbox{in} \quad \Omega_h \cap B_{k^{-1}},$$

 c) On $\p \Omega_h \cap B_{k^{-1}}$ we have
 $$|x_n-q_h(x')| \le \sigma |x'|^{2+\alpha}, \quad \quad |q_h(x')| \le \sigma,$$
 $$|u_h-p(x')|\le \sigma |x'|^{2+\alpha},$$
where $q_h$ is a quadratic polynomial.
\end{lem}

\begin{proof}
 By the localization theorem Theorem \ref{main_loc}, for all $h \le c$, $$S_h:=\{u<h\} \cap \bar \Omega,$$ satisfies $$k E_h \cap \bar \Omega  \subset S_h \subset k^{-1} E_h,$$ with
$$E_h=A^{-1}_hB_{h^{1/2}}, \quad \quad A_hx=x-\nu_hx_n$$
$$\nu_h \cdot e_n=0, \quad \|A_h^{-1}\|, \,\|A_h\| \le k^{-1} |\log h|.$$
Then we define $u_h$ as above and obtain $$S_1(u_h)=h^{-1/2}A_hS_h,$$ hence $$B_k \cap \bar \Omega_h \subset S_1(u_h) \subset B_{k^{-1}}^+,$$
where $$\Omega_h:= h^{-1/2}A_h \Omega.$$

Then  $$\det D^2 u_h=f_h(x)=f(h^{1/2}A_h^{-1}x),$$ and
\begin{align*}|f_h(x)-1| &\le M|h^{1/2} A_h^{-1} x|^\alpha \\
& \le M (h^{1/2}k^{-1}|\log h|)^\alpha |x|^\alpha \\
& \le \sigma |x|^\alpha
\end{align*}
if $h_0(M,\sigma)$ is sufficiently small.

Next we estimate $|x_n-h^{1/2}q(x')|$ and $|u_h-p(x')|$ on $\p \Omega_h \cap B_{k^{-1}}$. We have $$x \in \p \Omega_h \quad \Leftrightarrow \quad y:=h^{1/2} A_h^{-1} x \in \Omega,$$ or $$h^{1/2}x_n=y_n, \quad \quad h^{1/2} x'=y'-\nu_h y_n.$$
If $|x| \le k^{-1}$ then $$|y| \le k^{-1} h^{1/2} |\log h| |x| \le h^{1/4},$$ if $h_0$ is small hence, since $\Omega$ has an interior tangent ball of radius $\rho$, we have $$|y_n| \le \rho^{-1} |y'|^2.$$ Then $$|\nu_h y_n| \le k^{-1}|\log h| |y'|^2 \le |y'|/2,$$ thus $$\frac 12 |y'| \le |h^{1/2}x'|\le \frac 32 |y'|.$$ We obtain
\begin{align*}
|x_n-h^{1/2}q(x')| &\le h^{-1/2}|y_n-q(y')|+ h^{1/2}|q(h^{-1/2}y')-q(x')|\\
& \le M h^{-1/2}|y'|^{2+\alpha} + C h^{1/2}\left(|x'||\nu_h x_n|+|\nu_h x_n|^2\right) \\
& \le 2M h^{(\alpha+1)/2}|x'|^{2+\alpha} + C h^{1/2} \left (h^{1/2} |\log h||x'|^3+ h |\log h|^2|x'|^4 \right )\\
&\le \sigma |x'|^{2+\alpha},
\end{align*}
if $h_0$ is chosen small. Hence on $\p \Omega_h \cap B_{k^{-1}}$,
$$|x_n-q_h(x')| \le \sigma |x'|^{2+\alpha}, \quad \quad q_h:=h^{1/2}q(x'),$$
$$|q_h| \le \sigma,$$
and also
\begin{align*}
|u_h-p(x')| &\le h^{-1}|u(y)-p(y')|+|p(h^{-1/2}y')-p(x')|\\
&\le Mh^{-1}|y'|^{2+\alpha} + C \left(|x'||\nu_h x_n|+|\nu_h x_n|^2\right)\\
&\le 2M h^{\alpha/2}|x'|^{2+\alpha} + C \left (h^{1/2} |\log h||x'|^3+ h |\log h|^2|x'|^4 \right )\\
& \le \sigma |x'|^{2+\alpha}.
\end{align*}

\end{proof}

In the next lemma we show that if $\sigma$ is sufficiently small, then $u_h$ can be well-approximated by a quadratic polynomial near the origin.

\begin{lem}\label{lem2}
For any $\delta_0$, $\eps_0$ there exist $\sigma_0(\delta_0, \eps_0)$, $\mu_0(\eps_0)$ such that for any function $u_h$ satisfying properties a), b), c) of Lemma \ref{lem1} with $\sigma \le \sigma_0$ we can find
a rescaling $$\tilde u(x):=\frac{(u_h-l_h)(\mu_0 x)}{\mu_0^2},$$ with $$l_h(x)=\gamma_h x_n, \quad |\gamma_h| \le C_0,\quad \quad \mbox{$C_0$ universal},$$
that satisfies

a) in $\tilde \Omega \cap B_1,$

$$\det D^2 \tilde u=\tilde f, \quad \quad |\tilde f(x)-1| \le \delta_0 \eps_0 |x|^\alpha \quad \mbox{in} \quad \tilde \Omega \cap B_1,$$
and
$$|\tilde u-P_0|\le \eps_0 \quad \mbox{in} \quad \tilde \Omega \cap B_1,$$
for some $P_0$, quadratic polynomial,
$$\det D^2 P_0=1, \quad \|D^2 P_0\|\le C_0;$$

 b) On $\p \tilde \Omega \cap B_1$ there exist $\tilde p_0$, $\tilde q_0$ such that

 $$|x_n-\tilde q_0(x')| \le \delta_0 \eps_0 |x'|^{2+\alpha}, \quad \quad |\tilde q_0(x')| \le \delta_0 \eps_0,$$

and

 $$|\tilde u-\tilde p_0(x')|\le \delta_0 \eps_0 |x'|^{2+\alpha}, $$
$$\tilde p_0(x') =P_0(x'), \quad  \quad \frac \rho 2 |x'|^2 \le \tilde p_0(x') \le 2\rho |x'|^2.  $$
\end{lem}

\begin{proof}
 We prove the lemma by compactness. Assume by contradiction that the statement is false for a sequence $u_m$ satisfying a), b), c) of Lemma \ref{lem1} with $\sigma_m \to 0$. Then, we may assume after passing to a subsequence if necessary that
 $$p_m \to p_\infty, \quad q_m \to 0 \quad \mbox{uniformly on} \quad B_{k^{-1}},$$
 and $$u_m: S_1(u_m) \to \R$$ converges to (see Definition \ref{conv_def})$$u_\infty: \Omega_\infty \to \R. $$
 Then, by Theorem \ref{comp0}, $u_\infty$ satisfies
 $$ B_k^+ \subset \Omega_\infty \subset B_{k^{-1}}^+,  \quad \quad \det D^2u_\infty=1, $$

 $$\left \{
\begin{array}{l}
u_\infty=p_\infty(x') \quad \mbox{on} \quad \{p_\infty(x')<1\} \cap \{x_n=0\} \subset \p \Omega_\infty\\
u_\infty=1 \quad \quad \mbox{on the rest of} \quad \p \Omega_\infty.
\end{array}
\right.
$$

From Pogorelov estimate in half-domain (Theorem \ref{pe}) there exists $c_0$ universal such that
$$|u_\infty - l_\infty -P_\infty|\le c_0^{-1}|x|^3 \quad \mbox{in} \quad \mbox B_{c_0}^+,$$
where $$l_\infty:=\gamma_\infty x_n, \quad |\gamma_\infty|\le c_0^{-1},$$
and $P_\infty$ is a quadratic polynomial such that $$P_\infty(x')=p_\infty(x'), \quad \det D^2 P_\infty=1, \quad \|D^2 P_\infty\| \le c_0^{-1} .$$

Choose $\mu_0$ small such that $$c_0^{-1} \mu_0=\eps_0/32,$$ hence
$$|u_\infty-l_\infty-P_\infty| \le \frac 14 \eps_0 \mu_0^2 \quad \mbox{in} \quad B_{2 \mu_0}^+,$$
which together with $p_m \to p_\infty$ implies that for all large $m$
$$|u_m-l_\infty-P_\infty| \le \frac 12 \eps_0 \mu_0^2 \quad \mbox{in} \quad S_1(u_m) \cap B_{\mu_0}^+.$$

Then, for all large $m$, $$\tilde u_m:=\frac{(u_m-l_\infty)(\mu_0x)}{\mu_0^2}$$ satisfies in $\tilde \Omega_m \cap B_1$
$$|\tilde u_m-P_\infty| \le \eps_0 /2,$$
and
$$\det D^2 \tilde u_m =\tilde f_m(x)=f_m(\mu_0x),$$
$$|\tilde f_m(x)-1| \le \sigma_m (\mu_0|x|)^\alpha \le \delta_0 \eps_0 |x|^\alpha.$$
We define
$$\tilde q_m:=\mu_0 q_m, \quad \quad \tilde p_m:=p_m-\gamma_\infty q_m,$$
and clearly $$\tilde p_m \to p_\infty, \quad \tilde q_m \to 0 \quad \mbox{uniformly in} \quad B_1.$$

On $\p \tilde \Omega_m \cap B_1$ we have
\begin{align*}
|x_n-\tilde q_m(x')|&=\mu_0^{-1}|\mu_0x_n-q_m(\mu_0x')|\\
& \le \mu_0^{-1}\sigma_m |\mu_0x'|^{2+\alpha} \\
&\le \delta_0 \eps_0 |x'|^{2+\alpha},
\end{align*}
and
\begin{align*}
|\tilde u_m-\tilde p_m(x')| &= \mu_0^{-2}|(u_m-l_\infty)(\mu_0x)-p_m(\mu_0x')+\gamma_\infty q_m(\mu_0x')|\\
&\le \mu_0^{-2}(|(u_m-p_m)(\mu_0x')| +|\gamma_\infty||\mu_0x_n-q_m(\mu_0x')|)\\
&\le \sigma_m\mu_0^\alpha(1+|\gamma_\infty|)|x'|^{2+\alpha}\\
&\le \delta_0 \eps_0|x'|^{2+\alpha}.
\end{align*}

Finally, we let $P_m$ be a perturbation of $P_\infty$ such that $$P_m(x')=\tilde p_m(x'), \quad  \det D^2P_m=1, \quad P_m \to P_\infty \quad \mbox{uniformly in $B_1$.}$$

Then $\tilde u_m$, $P_m$, $\tilde p_m$, $\tilde q_m$ satisfy the conclusion of the lemma for all large $m$, and we reached a contradiction.
\end{proof}

From Lemma \ref{lem1} and Lemma \ref{lem2} we see that given any $\delta_0$, $\eps_0$ there exist a linear transformation $$T:=\mu_0 h_0^{1/2} A^{-1}_{h_0}$$ and a linear function $$l(x):=\gamma x_n$$ with $$|\gamma|, \|T^{-1}\|,\|T\| \le C(M,\delta_0, \eps_0),$$ such that the rescaling $$\tilde u(x):= \frac{(u-l)(Tx)}{(\det T)^{2/n}},$$ defined in $\tilde \Omega \subset \R^n$ satisfies

1) in $\tilde \Omega \cap B_1$ $$\det D^2 \tilde u= \tilde f, \quad |\tilde f-1|\le \delta_0 \eps_0 |x|^ \alpha,$$ and $$|\tilde u-P_0| \le \eps_0,$$ for some $P_0$ with $$\det D^2 P_0 =1, \quad \|D^2 P_0\| \le C_0;$$

2) on $\p \tilde \Omega \cap B_1$ we have $\tilde p$, $\tilde q$ so that

 $$|x_n-\tilde q(x')| \le \delta_0 \eps_0 |x'|^{2+\alpha}, \quad \quad |\tilde q(x')| \le \delta_0 \eps_0,$$

 $$|\tilde u-\tilde p(x')|\le \delta_0 \eps_0 |x'|^{2+\alpha}, \quad \quad \frac \rho 2 |x'|^2 \le \tilde p(x')=P_0(x') \le 2\rho |x'|^2.$$

 By choosing $\delta_0$, $\eps_0$ appropriately small, universal, we show in Lemma \ref{lem3} that there exist $\tilde l$, $\tilde P$ such that $$|\tilde u - \tilde l -\tilde P|\le C|x|^{2+\alpha}   \quad \mbox{in} \quad \tilde \Omega \cap B_1, \quad \quad \mbox{and} \quad |\nabla \tilde l|,\, \|D^2 \tilde P\| \le C,$$ with $C$ a universal constant. Rescaling back, we obtain that $u$ is well approximated by a quadratic polynomial at the origin i.e
 $$|u-l-P|\le C(M)|x|^{2+\alpha} \quad \mbox{in} \quad \Omega \cap B_\rho,\quad \quad \mbox{and} \quad |\nabla l|,\, \|D^2 P\| \le C(M)$$
 which, by \eqref{eq_u16}, proves Theorem \ref{c2alpha}.

  Since $\alpha \in (0,1)$, in order to prove that $\tilde u \in C^{2, \alpha}(0)$ it suffices to show that $\tilde u$ is approximated of order $2+\alpha$ by quadratic polynomials $l_m+P_m $ in each ball of radius $r_0^m$ for some small $r_0>0$, and then $\tilde l+\tilde P$ is obtained in the limit as $m \to \infty$ (see \cite{C2}, \cite{CC}). Thus Theorem \ref{c2alpha} follows from the next lemma.

\begin{lem}\label{lem3}
Assume $\tilde u$ satisfies the properties 1), 2) above. There exist $\eps_0$, $\delta_0$, $r_0$ small, universal, such that for all $m \ge 0$ we can find $l_m$, $P_m$ so that
  $$|\tilde u -l_m-P_m|\le \eps_0 r^{2+\alpha} \quad \mbox{in} \quad \tilde \Omega \cap B_r, \quad \quad \mbox{with} \quad r=r_0^m.$$
\end{lem}

\begin{proof}
We prove by induction on $m$ that the inequality above is satisfied with $$l_m=\gamma_m x_n, \quad |\gamma_m|\le 1,$$
  $$P_m(x')=\tilde p(x') -\gamma_m \tilde q(x'), \quad \det D^2 P_m=1, \quad \quad \|D^2 P_m\| \le 2 C_0.$$
  From properties 1),2) above we see that this holds for $m=0$ with $\gamma_0=0$.

Assume the conclusion holds for $m$ and we prove it for $m+1$. Let $$v(x):=\frac{(\tilde u-l_m)(rx)}{r^2}, \quad \mbox{with} \quad r:=r_0^m,$$ and define $$\eps:=\eps_0 r^\alpha.$$ Then
\begin{equation}\label{001}
|v-P_m|\le \eps \quad \mbox{in} \quad \Omega_v \cap B_1, \quad \Omega_v:=r^{-1}\tilde \Omega,
\end{equation}
$$|\det D^2 v -1| = |\tilde f(rx)-1| \le \delta_0 \eps.$$
On $\p \Omega_v \cap B_1$ we have
\begin{align*}
\left |\frac{x_n}{r}-\tilde q(x')\right|&=r^{-2}|rx_n-\tilde q(rx')|\\
&\le \delta_0 \eps |x'|^{2+\alpha},\\
& \le \delta_0 \eps,
\end{align*}
which also gives
\begin{equation}\label{001.5}
|x_n| \le 2 \delta_0 \eps \quad \mbox{on} \quad \p \Omega_v \cap B_1.
\end{equation}
From the definition of $v$ and the properties of $P_m$ we see that in $B_1$
$$|v-P_m|  \le r^{-2}|(\tilde u-\tilde p)(rx)|+|\gamma_m||x_n/r-\tilde q|+2nC_0|x_n|,$$
and the inequalities above and property 2) imply
\begin{equation}\label{002}
|v-P_m|\le C_1 \delta_0\eps \quad \mbox{in} \quad \p \Omega_v \cap B_1,
\end{equation}
with $C_1$ universal constant (depending only on $n$ and $C_0$).

We want to compare $v$ with the solution $$w:B_{1/8}^+ \to \R, \quad \quad \det D^2w=1,$$ which has the boundary conditions

$$\left \{
\begin{array}{l}
w=v \quad \mbox{on} \quad \p B_{1/8}^+ \cap \Omega_v\\
w=P_m \quad \quad \mbox{on} \quad \p B_{1/8}^+ \setminus \Omega_v.
\end{array}
\right.
$$

In order to estimate $|u-w|$ we introduce a barrier $\phi$ defined as
$$\phi:\bar B_{1/2} \setminus B_{1/4} \to \R, \quad \quad \phi(x):=c(\beta)\left(4^\beta-|x|^{-\beta} \right),$$ where $c(\beta)$ is chosen such that $\phi=1$ on $\p B_{1/2}$ and $\phi=0$ on $\p B_{1/4}$.

We choose the exponent $\beta>0$ depending only on $C_0$ and $n$ such that for any symmetric matrix $A$ with
$$(2C_0)^{1-n} I \le A \le (2C_0)^{n-1}I,$$
we have
$$Tr \,A(D^2\phi) \le -\eta_0 <0,$$
for some $\eta_0$ small, depending only on $C_0$ and $n$.

For each $y$ with $y_n=-1/4$, $|y'|\le 1/8$ the function
$$\phi_y(x):=P_m+\eps(C_1 \delta +\phi(x-y))$$
satisfies
$$\det D^2 \phi_y \le 1-\frac {\eta_0} {2} \eps  \quad \quad \mbox{in} \quad B_{1/2}(y)\setminus B_{1/4}(y),$$
if $\eps \le \eps_0$ is sufficiently small. From \eqref{001}, \eqref{002} we see that
$$v \le \phi_y \quad \mbox{on} \quad \p (\Omega_v \cap B_{1/2}(y)),$$
and if $\delta_0 \le \eta_0/2$,
$$\det D^2 v \ge \det D^2 \phi_y.$$
This gives
$$v \le \phi_y \quad \mbox{in} \quad \Omega_v \cap B_{1/2}(y),$$
and using the definition of $w$ we obtain
$$w \le \phi_y \quad \mbox{on} \quad \p B_{1/8}^+.$$
The maximum principle yields
$$w \le \phi_y \quad \mbox{in} \quad B_{1/8}^+,$$
and by varying $y$ we obtain

$$w(x) \le P_m +\eps(C_1 \delta_0+Cx_n) \quad \mbox{in} \quad B^+_{1/8}.$$
Recalling \eqref{001.5}, this implies $$w-P_m \le 2C_1 \delta_0 \eps \quad \mbox{on} \quad  B^+_{1/8}\setminus \Omega_v.$$
The opposite inequality holds similarly, hence
\begin{equation}\label{002.5}
|w-P_m|\le 2C_1 \delta_0 \eps \quad \mbox{on} \quad  B^+_{1/8}\setminus \Omega_v.
\end{equation}
From the definition of $w$ and \eqref{002} we also obtain
\begin{equation}\label{003}
|v-w| \le 3C_1 \delta_0 \eps \quad \mbox{on} \quad \p(\Omega_v \cap B^+_{1/8}).
\end{equation}

Now we claim that
\begin{equation}\label{004}
|v-w| \le C_2 \delta_0 \eps \quad \mbox{in} \quad \Omega_v \cap B^+_{1/8}, \quad \mbox{$C_2$ universal}.
\end{equation}
For this, we use the following inequality. If $A \ge 0$ is a symmetric matrix with $$1/2 \le \det A \le 2,$$ and $a \ge 0,$ then
\begin{align*}
\det(A+aI)&=\det A \det (I + a A^{-1})\\
& \ge \det A (1 + Tr (aA^{-1})) \\
& \ge \det A (1+ a/2)\\
&\ge \det A + a/4.
\end{align*}
This and \eqref{003} give that in $\Omega_v \cap B_{1/8}^+$
$$ w + 2 \delta_0 \eps (|x|^2- 2 C_1) \le v, $$
$$ v+ 2 \delta_0 \eps (|x|^2- 2 C_1) \le w,$$
and the claim \eqref{004} is proved.

Next we approximate $w$ by a quadratic polynomial near $0$. From \eqref{001},\eqref{002.5}, \eqref{004} we can conclude that
$$|w-P_m| \le 2 \eps \quad \mbox{in} \quad B_{1/8}^+,$$
if $\delta_0$ is sufficiently small. Since $w=P_m$ on $\{x_n=0\}$,  and
$$\frac \rho 4 |x'|^2 \le P_m(x') \le 4 \rho |x'|^2, \quad \det D^2 P_m=1, \quad \|D^2 P_m \| \le 2 C_0,$$
we conclude from Pogorelov estimate (Theorem \ref{pe}) that
$$\|D^2 w\|_{C^{1,1}(B_{c_0}^+)} \le c_0^{-1},$$
for some small universal constant $c_0$. Thus in $B_{c_0}^+$, $w-P_m$ solves a uniformly elliptic equation
$$Tr \, A(x)D^2(w-P_m)=0,$$ with the $C^{1,1}$ norm of the coefficients $A(x)$ bounded by a universal constant. Since $$w-P_m=0\quad \mbox{ on} \quad \{x_n=0\},$$ we obtain
$$\|w-P_m\|_{C^{2,1}(B_{c_0/2}^+)} \le C_3 \|w-P_m\|_{L^\infty(B_{c_0}^+)} \le 2C_3 \eps,$$ with $C_3$ a universal constant. Then
\begin{equation}\label{004.5}
|w-P_m-\tilde l_m-\tilde P_m|\le 2 C_3 \eps |x|^3 \quad \mbox{if $|x| \le c_0/2$},
\end{equation}
with $$\tilde P_m(x')=0,\quad  \tilde l_m=\tilde \gamma_m x_n, \quad \quad \quad |\tilde \gamma_m|,\,\|D^2 \tilde P_m\| \le 2C_3 \eps.$$
Since $\tilde l_m + P_m + \tilde P_m$ is the quadratic expansion for $w$ at $0$ we also have $$\det D^2(P_m + \tilde P_m) =1.$$

We define
$$P_{m+1}(x):=P_m(x) +\tilde P_m(x) -r \tilde \gamma_m \tilde q(x') +\sigma_m x_n^2$$
with $\sigma_m$ so that $$\det D^2 P_{m+1}=1,$$
and let
$$l_{m+1}(x):=\gamma_{m+1}x_n, \quad \quad \gamma_{m+1}=\gamma_m+ r \tilde \gamma_m.$$

Notice that
\begin{equation}\label{005}
|\gamma_{m+1}-\gamma_m|,\, \, \|D^2 P_{m+1}- D^2 P_m\| \, \le C_4 \eps=C_4\eps_0r_0^{m\alpha},
\end{equation}
and
$$ \|D^2 P_{m+1}- D^2 (P_m+ \tilde P_m)\| \le C_4 \delta_0 \eps, $$ for some $C_4$ universal. From the last inequality and \eqref{004}, \eqref{004.5} we find
$$|v-\tilde l_m -P_{m+1}| \le (2C_3 r_0^3+C_2 \delta_0 + C_4 \delta_0) \eps \quad \mbox{in} \quad \Omega_v \cap B_{r_0}^+.$$
This gives
$$|v-\tilde l_m -P_{m+1}| \le \eps r_0^{2+\alpha} \quad \mbox{in} \quad \Omega_v \cap B_{r_0}^+,$$
if we first choose $r_0$ small (depending on $C_3$) and then $\delta_0$ depending on $r_0$, $C_2$, $C_4$,
hence $$|\tilde u- l_{m+1}-P_{m+1}| \le \eps r^2 r_0^{2+\alpha}=\eps_0(rr_0)^\alpha \quad \mbox{in} \quad \Omega \cap B^+_{rr_0}.$$
Finally we choose $\eps_0$ small such that \eqref{005} and $$\gamma_0=0, \quad \|D^2P_0\| \le C_0,$$ guarantee that $$|\gamma_m| \le 1, \quad  \|D^2 P_m\| \le 2 C_0$$ for all $m$. This shows that the induction hypotheses hold for $m+1$ and the lemma is proved.

\end{proof}

 \begin{rem}\label{rem7}The proof of Lemma \ref{lem3} applies also at interior points. More precisely, if $\tilde u$ satisfies the hypotheses in $B_1(x_0) \subset \tilde \Omega$ instead of $B_1 \cap \tilde \Omega$ then the conclusion holds in $B_1(x_0)$. The proof is in fact simpler since, in this case we take $w$ so that $$w=v \quad \mbox{on} \quad \p B_1(x_0),$$ and then \eqref{003} is automatically satisfied, so there is no need for the barrier $\phi$. Also, at the end we apply the classical interior estimate of Pogorelov instead of the estimate in half-domain.
\end{rem}

Now we can sketch a proof of Corollary \ref{cor1} and Theorem \ref{global}. 
 
 If $u$ satisfies the conclusion of Theorem \ref{c2alpha} then, after an appropriate dilation, any point in $\mathcal{C}_\theta \cap B_\delta$ becomes an interior point $x_0$ as in Remark \ref{rem7} above for the rescaled function $\tilde u$. Moreover, the hypotheses of Lemma \ref{lem3} hold in $B_1(x_0)$ for some appropriate $\eps \le \eps_0$. Then Corollary \ref{cor1} follows easily from Remark \ref{rem7}.
 
 If $u$ satisfies the hypotheses of Theorem \ref{global} then we obtain as above that
 $$\|u\|_{C^{2,\alpha}(D_\delta)} \le C, \quad \quad D_\delta:=\{x\in \Omega| \quad dist(x,\p \Omega) \le \delta\},$$
for some $\delta$ and $C$ depending on the data. We combine this with the interior $C^{2,\alpha}$ estimate of Caffarelli in \cite{C2} and obtain the desired bound.

\end{document}